\documentclass[10pt]{article}

\usepackage{amsmath, amsthm, amsfonts,amssymb}
\usepackage{graphicx}
\usepackage{colortbl}
\usepackage[utf8]{inputenc}

\usepackage[colorlinks=true,linkcolor=blue,citecolor=blue,urlcolor=blue,breaklinks]{hyperref}

\def\R{\mathbb{R}}

\def\NN{\mathbb{N}}

\def\ird{\int_{\mathbb{R}^N}}

\def\d{\,{d}}
\def\dx{\,{d}x}
\def\dy{\,{d}y}
\def\dz{\,{d}z}
\def\dv{\,{d}v}
\def\dw{\,{d}w}

\newtheorem{thm}{Theorem}[section]

\newtheorem{lem}[thm]{Lemma}

\theoremstyle{definition}

\theoremstyle{remark}
\newtheorem{remark}[thm]{Remark}
\theoremstyle{example}

\author{Alethea B. T. Barbaro\footnote{Department of Mathematics,
    Applied Mathematics \& Statistics, Case Western Reserve
    University, 10900 Euclid Avenue, Yost Hall, Cleveland, Ohio
    44106-7058,
    USA. \texttt{\href{mailto:alethea.barbaro@case.edu}{alethea.barbaro@case.edu}}
  }
  \and
  Jos\'e
  A. Ca\~nizo\footnote{Departamento de Matem\'atica Aplicada,
    Universidad de Granada, Facultad de Ciencias, Campus de
    Fuentenueva, 18071 Granada,
    Spain. \texttt{\href{mailto:canizo@ugr.es}{canizo@ugr.es}}}
  \and
  Jos\'e
  A. Carrillo\footnote{Department of Mathematics, Imperial College
    London, London SW7 2AZ, United Kingdom.
    \texttt{\href{mailto:carrillo@imperial.ac.uk}{carrillo@imperial.ac.uk}}
  }
  \and
  Pierre Degond\footnote{Department of Mathematics, Imperial
    College London, London SW7 2AZ, United
    Kingdom.
    \texttt{\href{mailto:pdegond@imperial.ac.uk}{pdegond@imperial.ac.uk}}
  }}

\title{Phase Transitions in a kinetic flocking model of Cucker-Smale type}
\date{October 10, 2015}

\makeindex

\begin{document}

\maketitle

\begin{abstract}
  We consider a collective behavior model in which individuals try to
  imitate each others' velocity and have a preferred speed. We show
  that a phase change phenomenon takes place as diffusion decreases,
  bringing the system from a ``disordered'' to an ``ordered''
  state. This effect is related to recently noticed phenomena for the
  diffusive Vicsek model. We also carry out numerical simulations of
  the system and give further details on the phase transition.
\end{abstract}

\tableofcontents

\section{Introduction}
\label{S:intro}

In many biological systems made of a large number of individuals such
as cell populations \cite{Yamao_etal_PLOSone11}, insect colonies
\cite{Buhl_etal_Science2006} or vertebrate groups
\cite{Katz_etal_PNAS11}, agents are known to strongly interact with
each other.  Such social forces trigger the emergence of collective
dynamics, where all individuals are moving coherently.  Furthermore,
it has been observed that such collective dynamics necessitate
specific circumstances, such as e.g. a large density of individuals
\cite{Buhl_etal_Science2006} while the same system exhibits
disorganized dynamics when these circumstances are not met.  Hence,
when considering any sort of self-organized system, the question of
how the model switches between disorganized and collective behavior
becomes of tantamount importance.

Models of collective behavior abound in the literature, at the
particle
\cite{Aoki,bebsvps2009,Couzin_etal_JTB02,DCBC,gc2004,vcbcs1995},
kinetic
\cite{Bertin_etal_PRE06,bcc2011,ccr2011,CFTV,Fornasier_etal_PhysicaD11}
and hydrodynamic levels
\cite{Bertin_etal_JPA09,dm2008,Toner_Tu_PRL95}. Among these models,
the Vicsek model (VM) \cite{bcc2012,dm2008,vcbcs1995} is one of the
simplest models exhibiting phase transitions between disordered to
collective dynamics. In this model, particles moving with constant
speed interact with their neighbors through local alignment and are
subject to noise. Another extremely successful model is the
Cucker-Smale model (CSM)
\cite{Carrillo_etal_SIMA10,Cucker_Smale_IEEE07,Ha_Liu_CMS09,Ha_Tadmor_KRM08}. In
this model, particles can take all possible speeds and interact
through local velocity consensus in a deterministic way, although
noisy versions have also been considered
\cite{bcc2011,Cucker_Mordecki_JMPA08}.

In its original version, the CSM does not feature any
self-propulsion. Particles move just because of inertia. In
particular, if the initial total momentum of the particle system is
equal to zero and in the absence of noise, particles become still in
the large time limit. By constraining velocities to belong to a
sphere, the VM exhibits a significantly different behavior, with phase
transitions from disordered to collective states made possible
\cite{vcbcs1995}. This is why it is natural to equip the CSM with
self-propulsion as proposed e.g. in \cite{bd2012, bcc2011} and to
examine if such an augmented CSM exhibits similar phase transitions as
the VM. Note that when the strength of the self-propulsion term is let
to infinity, the noisy, self-propelled CSM converges to the VM as
proven in \cite{bostan2013asymptotic}. Therefore, one may guess that
this augmented CSM features phase transitions in at least some range
of parameters. It is the goal of the present paper to prove this fact.

Phase transitions for the kinetic VM have been extensively studied in
\cite{dfl2013, dfl2015,frouvelle2012dynamics} in the spatially
homogeneous case. Remember that the velocities are normalized, such
that the velocity distribution is a probability distribution on the
sphere. Supposing that the density is scaled to unity, then it is
shown that there is a critical noise value $D_c$ of the noise
intensity $D$ such that for $D>D_c$, only isotropic stationary
distributions exist and are stable. By contrast, when $D<D_c$,
isotropic stationary distributions become unstable and a family of
non-isotropic equilibria parametrized by a unit vector $\Omega$ on the
sphere emerge and are stable. These equilibria are given by von
Mises-Fisher distributions which are substitutes for the Gaussian
distribution for probabilities defined on the sphere
\cite{Watson_JAP82}. Note that \cite{dfl2013,
  dfl2015,frouvelle2012dynamics} provide fully nonlinear (in)stability
results which are obtained by taking advantage of the variational
structure of the kinetic VM induced by a conveniently defined free
energy functional.

Now considering the noisy, self-propelled CSM, the first result
towards the emergence of a phase transition is given in
\cite{bd2012}. In this model, a subterfuge is used by considering a
time-scale separation. For this purpose, the force acting on the
particle is decomposed into the self-propulsion force on the one hand,
and the force deduced from the combination of alignment with the
neighbors and noise on the other hand. Each of these two forces is
scaled with parameters $\nu$ and $\mu$ respectively. Now, considering
again a spatially homogeneous situation for simplicity, \cite{bd2012}
considers the limit $\mu \to \infty$ keeping $\nu$ fixed (i.e. when
the alignment+noise contribution of the force is large compared to the
self-propulsion force). In this limit, under a convenient time
rescaling, the velocity distribution is shown to converge to a
Maxwellian distribution whose mean velocity obeys an ordinary
differential equation modelling the action of the self-propulsion
force. Again, the number of equilibria that this ODE possesses depends
on the noise intensity. Consistently with what was found for the VM,
there again exists a threshold noise $D_c$ above which the only
equilibrium mean-velocity is $0$ and is stable, showing that no
collective motion emerges. By contrast, for values of the noise
intensity below $D_c$, the zero mean-velocity equilibrium becomes
unstable and a whole sphere of stable equilibria for the mean velocity
emerges. The time-scale separation allows for an analytic
determination of the equilibria of the interaction term (here
alignment+noise), which turns out to be the Maxwellian. In particular,
in the stationary states the mean velocity of this Maxwellian can be
explicitly computed as a function of the noise. Without this scale
separation hypothesis, the equilibria are given by a more complicated
formula and their mean velocity is not explicitly known but is rather
a solution of a nonlinear equation, making the analysis more
difficult. The task tackled here is to provide a rigorous analysis of
this equation.

In this work, we consider a noisy, self-propelled CSM.  In this model,
$f$ is the distribution in both space $x$ and velocity $v$ at time
$t$, and the model features a CSM term which aligns the velocity of
individuals nearby in space, a term adding noise in the velocity, and
a friction term which relaxes velocities back to norm one:
$$
\partial_t f + v \nabla_x f = \nabla_v \cdot \left( \alpha (|v|^2-1)v f + (v-u_f) f + D \nabla_v f \right)\,.
$$
where
$$  u_f (t,x) = \frac{\int K(x,y) v f(t,y,v) \dv \dy}{\int K(x,y) f(t,y,v) \dv \dy}.
$$
Here $K(x,y)$ is a suitably defined localization kernel and $\alpha$
and $D$ are respectively the self-propulsion force and noise
intensities. We have chosen scales such that the alignment force
(modelled by the term $(v-u_f) f$) has intensity equal to $1$.  In
this work, we focus on the spatially homogeneous case, where the model
reduces to
\begin{align}\label{E:mainPDE}
\partial_t f &= \nabla_v \cdot \left( \alpha (|v|^2-1)v f + (v-u_f) f + D \nabla_v f \right)\,.
\end{align}
where
\begin{equation} \label{E:u_f-intro}
  u_f (t) = \frac{\int v f(t,v) \dv}{\int f(t,v) \dv},
\end{equation}
and where $f=f(t,v)$ is the velocity distribution at time $t$.
Precisely, the goal of this work is to show that there is a phase
transition between unpolarized and polarized motion as the noise
intensity $D$ is varied, for a specific range of the values of
$\alpha$.  Therefore we achieve the goal of proving that the noisy
self-propelled CSM behaves like the VM when the self-propulsion speed
is large enough.

Note that this problem has already been studied in dimension
1. Specifically, the existence of a transition between one and three
stationary states has been proven in a one-dimensional setting by
Tugaut in \cite{Tugaut2013Convergence, Tugaut2013Phase,
  Tugaut2013Selfstabilizing,Tugaut2014Selfstabilizing}. Here, we are
interested in an arbitrary number of dimensions (practically $2$ or
$3$) and we numerically demonstrate that there is indeed a phase
transition by verifying that the stability of the isotropic equilibria
changes as $D$ crosses a threshold value.  In addition, we
analytically prove that for large noise $D$ there is only one
isotropic stationary solution, while for small $D$ there is an
additional infinite family of stationary states parametrized by a unit
vector on the sphere, below referred to as the polarized
equilibria. Therefore, although no analytic formula for the mean
velocity of the polarized equilibria of the CSM exist, we know that at
least this family of equilibria forms a manifold of the same dimension
as the polarized equilibria of the VM. We remind that, in the VM case,
the polarized equilibria were given by Von Mises-Fisher distributions
about an arbitrary mean orientation $\Omega$ belonging to the unit
sphere. Here, these polarized equilibria are still parametrized by a
vector $\Omega$ of the unit sphere, but the precise expression of the
mean velocity is not analytically known.

The paper is organized as follows.  In Section
\ref{S:homogeneousProblem}, we introduce the homogeneous problem and
discuss its steady states.  We then focus on multiple dimensions and
postulate that there are two regions of the parameter space, each with
a different number of possible stationary solutions.  We proceed to
show that for small $D$, there is a manifold of equilibria
parametrized by a vector on the unit sphere, while for large $D$,
there can be only one. In section \ref{S:numerics}, we numerically
validate the results of Section \ref{S:homogeneousProblem} and explore
the influence of the parameter $\alpha$ and the stability of
stationary states. We conclude and discuss future work in section
\ref{S:conclusion}. Some technical details are provided in Appendix
\ref{sec:appendix}.


\section{Phase Transition: Stationary States for the Homogeneous
  Case}
\label{S:homogeneousProblem}

In this section we focus on probability density solutions to the
spatially homogeneous kinetic model associated to \eqref{E:mainPDE},
that reads
\begin{equation}\label{E:main}
\partial_t f = \nabla_v \cdot \left( \alpha v(|v|^2-1)f + (v-u_f) f \right) + D \Delta_v f\,,
\end{equation}
with $t\geq 0$ and $v\in \R^N$. Here, the mean velocity $u_f$ is given by
\begin{equation} \label{E:u_f}
  u_f (t) = \int_{\R^N} v f(t,v) \dv\,.
\end{equation}
The second term on the right-hand side of \eqref{E:main} accounts for the tendency to align with the local velocity field, while the last term adds noise into the velocity component. The first term enforces the tendency to travel with unit speed. The kinetic equation satisfies the conservation of mass 
$$
\int_{\R^N} f(t,v) \dv =1\,,
$$
for all $t\geq 0$. This equation lies in the general class of PDEs
having a gradient flow structure, see
\cite{MR2053570}, by writing the equation as
$$
\partial_t f = \nabla_v \cdot \left( f\nabla_v \xi \right) \qquad \mbox{with } \xi=\Phi(v)+W\ast f + D \log f\,.
$$
Here, particles are thought to move under the effect of a confining potential given by 
$$
\Phi(v)=\alpha \left(\tfrac{|v|^4}{4} -\tfrac{|v|^2}{2}\right)\,,
$$ 
an interaction potential of the form $W(v)=\tfrac{|v|^2}2$, and with linear diffusion. The equation \eqref{E:main} is then seen as a continuity equation with a velocity field of the form $-\nabla_v \xi$, and thus there is a natural entropy for this equation given by the free energy of the system
\begin{align}
\mathcal{F}[f]:= &\,\int_{\R^N} \Phi(v) f(t,v) \dv+ \frac12
                   \int_{\R^N} \int_{\R^N} W(v-w) f(t,v) f(t,w) \dw \dv + D \int_{\R^N}  f \log f(v) \dv\nonumber \\
=&\,\int_{\R^N} \left(\alpha \tfrac{|v|^4}{4} +(1-\alpha) \tfrac{|v|^2}{2}\right) f(t,v) \dv - \frac12 |u_f|^2 + D \int_{\R^N}  f \log f(v) \dv\,,\label{freeen}
\end{align}
since $\xi = \frac{\delta \mathcal{F}}{\delta f}$. The second expression follows by expanding the square in the interaction potential. Actually,  the dissipation of the free energy $\mathcal{F}[f]$ in \eqref{freeen} along solutions is given by
$$
-\frac{d\mathcal{F}[f(t)]}{dt} = D[f(t)] := \int_{\R^N} \left|\nabla_v \xi\right|^2 f(t,v) \dv\,.
$$

We will look for stationary solutions $f(v)>0$ to \eqref{E:main} for
$u_f=\bar u$. Taking into account the dissipation property, they
should satisfy $\nabla_v \xi=0$, or equivalently
$\xi = \text{constant}$. Thus, stationary solutions are of the form
\begin{equation}
  \label{stateq}
  f_{\bar u}(v)
  = \frac{1}{Z}
  \exp\left( -\tfrac{1}{D}
    \left[ \alpha \tfrac{|v|^4}{4} +(1-\alpha) \tfrac{|v|^2}{2}
      - \bar u \cdot v
    \right]
  \right),
\end{equation}
with $Z$ the normalization factor such that $f$ has unit
mass. Therefore, the set of stationary solutions of \eqref{E:main} can
be parametrized by the set of mean velocities $\bar u \in \R^N$ such
that
\begin{equation}
  \label{defcalH}
  \mathcal{H}({\bar{u}},D) := \int_{\R^N} (v-{\bar{u}})f_{\bar u}(v) dv =0.
\end{equation}
Notice that by \eqref{defcalH}, $\mathcal{H}({\bar{u}},D)=0$ is
equivalent to $\bar u=u_{f_{\bar u}}$ given by \eqref{E:u_f}. Let us
point out that $\bar u=0$ is always a solution corresponding to
radially symmetric stationary states.

\begin{figure}[ht]
\begin{center}
\includegraphics[width=16cm]{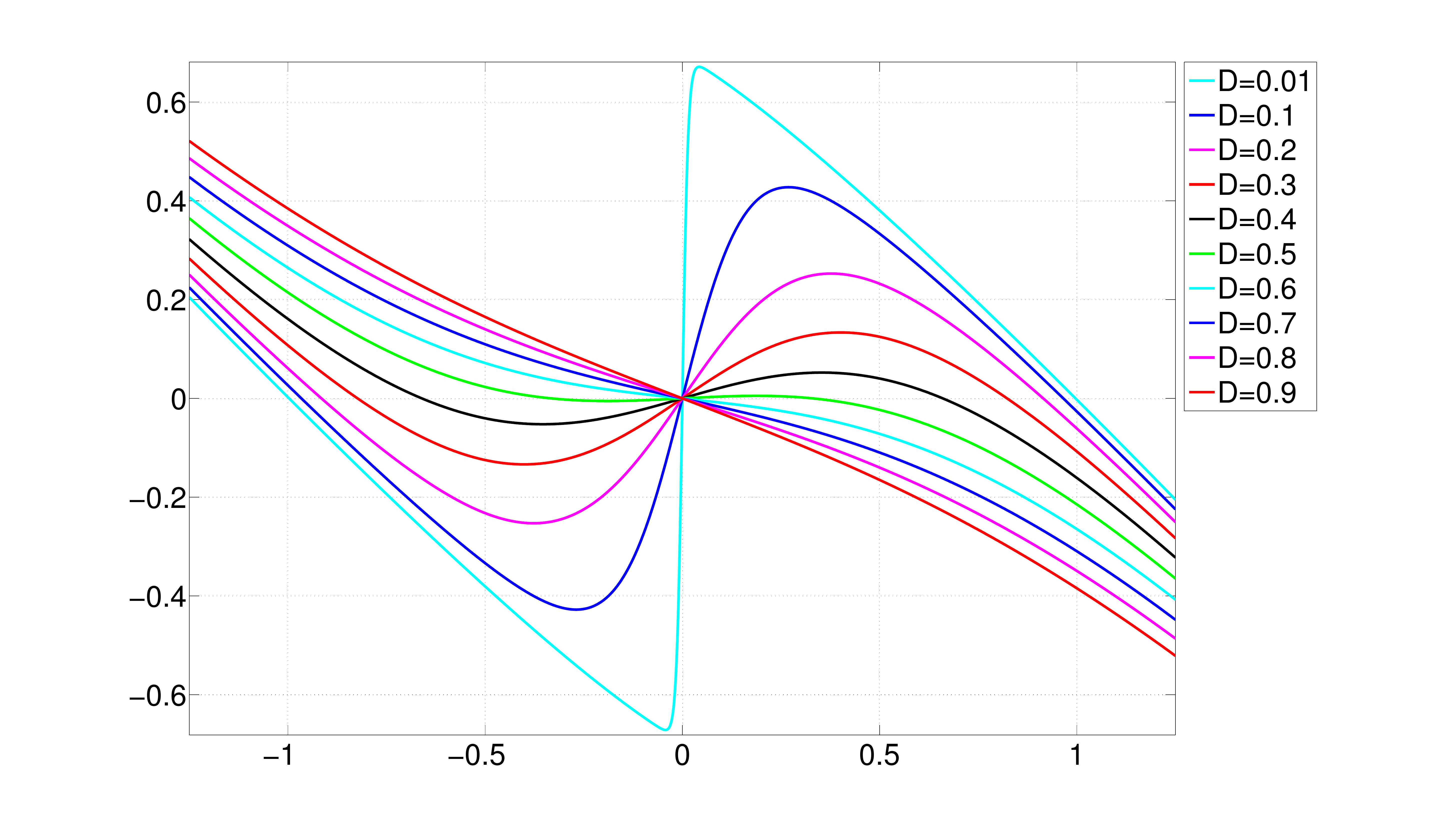}
\end{center}
\caption{Plot of the one-dimensional $H(u,D)$ against $u$ for $\alpha = 2$ and varying values of $D$.  From the figure, it is apparent that the sign of $\frac{\partial H}{\partial u}(0)$ shifts from negative to positive as $D$ increases.}
\label{F:Hofu}
\end{figure}

By choosing the axis, we may assume without loss of generality that
$\bar u$ points in the direction of the first axis or first vector
$e_1$ of the canonical basis, and then let us denote the magnitude of
$\bar u$ by $u\geq 0$. The full set of stationary solutions is
obtained by composing $f_{\bar u}$ with any rotation in $\R^N$, and
thus yields an $(N-1)$-dimensional family of stationary solutions for
each $\bar u=u e_1$ satisfying \eqref{defcalH}. Noticing that all
components of $\mathcal{H}$ except for the first one vanish due to
$f_{\bar u}(v)$ being odd in $v_2, \dots, v_N$, we can restrict our
attention to the first component of $\mathcal{H}$. For the sake of
simplicity we will denote by $f$ the probability density given by
\eqref{stateq} associated to the vector $\bar u=u e_1$, and the real
valued function whose roots have to be analyzed is the first component
of $\mathcal{H}$, given by
\begin{equation}
  \label{defH}
  H(u,D) = \int_{\R^N} (v_1-u)f(v) dv
  = \frac1{Z} \int_{\R^N} (v_1-u) \exp\left\{-\frac{P_u(v)}{D}\right\} \dv\,,
\end{equation}
with
\begin{align}\label{poly}
  P_u(v) &= -\alpha \left( \tfrac{|v|^2}{2} - \tfrac{|v|^4}{4} \right)
           + \tfrac{|v|^2}{2} - uv_1\, .
\end{align}
In figure \ref{F:Hofu}, we plot $H(u,D)$ in one dimension as a
function of $u$ for varying values of $D$.  It is clear from the
figure that for small values of $D$, $H(u,D)$ has three roots, the
zero root and two roots with identical speed; while for large values
of $D$ the only root is $u=0$, and this can be deduced from the sign
of $\frac{\partial H}{\partial u}(0,D)$.  We will show analytically
that this is the case in the next subsections and it will be further
studied numerically in section \ref{S:numerics}.

Our goal is to show that, for given $\alpha>0$, as we vary the noise
strength $D$, there is a region of the parameter space with only one
possible solution, namely $u=0$, and a region with at least two roots,
$u=0$ and $u=u_{\alpha, D} > 0$.  In fact, we expect to have the
unique homogeneous stationary state for large noise corresponding to a
disordered state while for small noise we expect to have a nontrivial
biased solution. The objective of the next two subsections is to show
these facts for small and large noise.

The main theorem of this section can be summarized as follows:

\begin{thm}{\bf (Phase Transition Driven by Noise)}\label{main} 
  The nonlinear Fokker-Planck equation \eqref{E:main} exhibits a phase transition in the following
  sense:
  \begin{enumerate}
  \item For small enough diffusion coefficient $D$ there is a function
    $u = u(D)$ with $$\lim_{D\to 0}u(D) = 1,$$ such that $f_{\bar{u}}$ given
    by \eqref{stateq} with $\bar{u} = (u(D),0,\dots,0)$ is a
    stationary solution of \eqref{E:main}.
    
  \item For large enough diffusion coefficient $D$ the only stationary
    solution of \eqref{E:main} is the symmetric distribution given by
    \eqref{stateq} with $\bar{u} = 0$.
  \end{enumerate}
\end{thm}

Let us notice that the previous theorem does prove the appearance of a
phase transition, although it does not give information about the critical value where it
occurs. We will numerically show in section \ref{S:numerics} that this
phase transition is continuous and it happens at a sharp value of $D$
as in \cite{dfl2013} for the continuous Vicsek model. We will numerically
demonstrate in subsection \ref{S:stability} the time asymptotic stability
of the spatially homogeneous solution for large noise and the non homogeneous
solution for small noise in one dimension using a Monte Carlo-like particle method.


\subsection{$D \to \infty$ Case: Unique Disordered State}
\label{sec:Dinfty}

We define $\tilde H(u,D)=Z\,H(u,D)$ as
\begin{equation*}
  \tilde H(u,D) =
  \int_{\R^N} (v_1-u) \exp \left( -\tfrac{|v|^2}{2D} +
    \tfrac{uv_1}{D} \right) \exp\left( -\alpha \left( \tfrac{|v|^4}{4D}
      - \tfrac{|v|^2}{2D} \right)\right) dv.
\end{equation*}
Noticing that
\begin{equation*}
  (u-v_1) = D \frac{\partial}{\partial v_1} \left( -\frac{|v|^2}{2D} + \frac{u v_1}{D} \right)\,,
\end{equation*}
we may integrate by parts and obtain
\begin{align*}
  \tilde H(u,D) 
  &=
  - D \int_{\R^N} \left(
    \frac{\partial}{\partial v_1} \exp \left( -\tfrac{|v|^2}{2D} +
      \tfrac{uv_1}{D} \right)
  \right)
  \exp\left( -\alpha \left( \tfrac{|v|^4}{4D}
      - \tfrac{|v|^2}{2D} \right)\right) dv
  \nonumber\\
  &= -\alpha
  \int_{\R^N}
  \exp \left( -\tfrac{|v|^2}{2D} +
    \tfrac{uv_1}{D} \right)
  \exp\left( -\alpha \left( \tfrac{|v|^4}{4D}
      - \tfrac{|v|^2}{2D} \right)\right)
  v_1(|v|^2-1) dv.
\end{align*}

\begin{lem} 
  There exists $D_0 > 0$ such that
  $\frac{\partial \tilde H}{\partial u} < 0$ for all $u>0$ and $D \geq D_0$.
\end{lem}
\begin{proof}
Computing $\frac{\partial \tilde H}{\partial u}$, we get
\begin{align*}
\frac{\partial \tilde H}{\partial u} &= \frac{\alpha}{D} \int_{\R^N} \exp \left( -\tfrac{|v|^2}{2D} + \tfrac{u v_1}{D} \right)  \exp\left( -\alpha \left( \tfrac{|v|^4}{4D} - \tfrac{|v|^2}{2D} \right)\right) v_1^2(1-|v|^2) \dv\,.
\end{align*}
The term $1-|v|^2$ obviously determines the sign of $\frac{\partial \tilde H}{\partial u}$.  We will compensate the positivity of this term on the unit ball with a piece of the integral outside.  First, let us estimate the integrand inside the unit ball,
\begin{equation}\label{tech1}
\exp \left( -\tfrac{|v|^2}{2D} + \tfrac{u v_1}{D} \right)  v_1^2(1-|v|^2) \leq  \exp(\frac{u}{D})\,.
\end{equation}
We will also use that the other terms are close to 1 in any bounded region for $D$ large enough. More precisely, for all $\epsilon >0$  
\begin{equation}\label{tech2}
\left|\exp \left( -\alpha \left( \tfrac{|v|^4}{4D} - \tfrac{|v|^2}{2D} \right)\right) -1\right|\leq \varepsilon\,,
\end{equation}
for $D\geq D(\epsilon)$ large enough and $|v|<4$. 

Let us write $\R^d=\mathcal{A}\cup \mathcal{B}\cup \mathcal{C}$ with $\mathcal{A}=B_1(0)$, $\mathcal{B}=B_1(\eta)$ and $\mathcal{C}=\R^d\setminus (\mathcal{A}\cup\mathcal{B})$. Here, the notation $\eta$ refers to the vector $(3,0,\dots,0)$ and $B_1(\tilde\eta)$ denotes the Euclidean ball of radius 1 centered at $\tilde\eta$. We separate the integrand into three pieces corresponding to the sets $\mathcal{A}$, $\mathcal{B}$, and $\mathcal{C}$. Since the integrand is negative in $\mathcal{C}$, then we can estimate
\begin{align*}
\frac{\partial \tilde H}{\partial u} \leq  I+II\,,
\end{align*}
where $I$ and $II$ are the integrals restricted to $\mathcal{A}$ and $\mathcal{B}$ respectively. Taking into account \eqref{tech1} and \eqref{tech2}, we control the first term as
\begin{align*}
I &:=  \frac{\alpha}{D} \int_{\mathcal{A}} \exp \left( -\tfrac{|v|^2}{2D} + \tfrac{u v_1}{D} \right)  \exp\left( -\alpha \left( \tfrac{|v|^4}{4D} - \tfrac{|v|^2}{2D} \right)\right) v_1^2(1-|v|^2) dv \leq \frac{\alpha}{D}  |\mathcal{A}| (1 + \epsilon) \exp\left(\frac{u}{D}\right).
\end{align*}
Similarly, the second term is bounded by
\begin{align*}
II &= -\frac{\alpha}{D} \int_{\mathcal{B}} \exp \left( -\tfrac{|v|^2}{2D} + \tfrac{u v_1}{D} \right)  \exp\left( -\alpha \left( \tfrac{|v|^4}{4D} - \tfrac{|v|^2}{2D} \right)\right) v_1^2(|v|^2-1) dv\\
&\leq -\frac{\alpha}{D} \int_{\mathcal{B}} \exp \left( -\tfrac{|v|^2}{2D} + \tfrac{u v_1}{D} \right)  (1-\epsilon) v_1^2(|v|^2-1) dv \leq -\frac{12\alpha}{D} |\mathcal{B}|(1-\epsilon) \exp \left( -\tfrac{16}{2D} + \tfrac{2u}{D} \right)
\end{align*}
due to \eqref{tech2} and since $v_1>2$ and $|v|^2\leq 4$ in $\mathcal{B}$. In order to show that $\frac{\partial \tilde H}{\partial u} \leq 0$, we need only show that $I+II \leq 0$:
\begin{align*}
\frac{\partial \tilde H}{\partial u} \leq  I+II &\leq \frac{\alpha}{D} |\mathcal{A}| \left( (1 + \epsilon) \exp\left(\frac{u}{D}\right) -  12(1-\epsilon) \exp \left( -\tfrac{8}{D} + \tfrac{2u}{D} \right)\right)\\
&\leq \frac{\alpha}{D} |\mathcal{A}| \exp\left(\frac{2u}{D}\right) \left( (1 + \epsilon) -  12(1-\epsilon) \exp \left( -\tfrac{8}{D} \right)\right)
\,.
\end{align*}
Here, we use that $|\mathcal{A}|=|\mathcal{B}|$. Choosing $\epsilon=\tfrac12$, we see that $\frac{\partial \tilde H}{\partial u}\leq 0$ for $D\geq\max( D(\tfrac12),\frac{8}{\log 12})$.
\end{proof}

The preceding claim proves that $H(u,D)$ can have at most one nonnegative root, and since $H(0,D)=0$, there can be no other positive root.  Hence, in the case of $D \to \infty$, there can only be the radially symmetric stationary solution associated to $u=0$.


\subsection{Laplace's Method}
\label{S:laplacesMethod}

Laplace's method gives the asymptotics of integrals of the form
\begin{equation*}
  \int_{\R^N} e^{\frac{f(x)}{D}} g(x) \dx
\end{equation*}
for given functions $f(x)$ and $g(x)$ as $D \to 0$. A usual statement
of Laplace's Method that is commonly found in the asymptotic analysis
literature (see \cite{BO,miller} for instance) reads as follows:

\begin{thm}
  Assume $f$ is a smooth function that has a unique global minimum at
  $x_0$, and that there exist $\epsilon, \delta > 0$ such that
  $f(x) > f(x_0) + \delta$ for all $x$ with $|x-x_0| > \epsilon$.
  Then, as $D \to 0$,
\begin{equation*}
  \int_{\R^N} e^{\frac{-f(x)}D} g(x) dx
  \sim
  \left( \pi D \right)^\frac{N}{2}
  \det( \Lambda(x_0))^{-\frac{1}{2}}
  \,e^{\frac{-f(x_0)}D} g(x_0),
\end{equation*}
where $\Lambda(x_0)$ is the Hessian matrix of $f$ at $x_0$.
\end{thm}

We state and briefly prove a modified version of it including higher
order terms which is well adapted to our arguments, avoiding general
statements which become cumbersome in a multidimensional setting. We
essentially follow the strategy in \cite[Chapter 6]{BO}; see also
\cite{miller} for multidimensional statements. For a multi-index
$\beta = (\beta_1,\dots,\beta_N)$ we denote
\begin{equation*}
  M_\beta := \ird e^{-|x|^2} x^\beta \dx
  =  \ird e^{-|x|^2} x_1^{\beta_1} \cdots x_N^{\beta_N} \dx.
\end{equation*}
Of course, the constant $M_\beta$ is $0$ whenever one of the
$\beta_i$ is odd, and $M_\beta$ may also be expressed in terms of
the Gamma function. The following calculation is at the basis of Laplace's method:

\begin{lem}
  \label{lem:Laplace-basic}
  Let $\beta = (\beta_i)_{i=1,\dots,N}$ be a multi-index and
  Let $Q: \R^N \to \R$ be the quadratic function given by
  \begin{equation*}
    Q(x) = Q(x_1,\dots,x_N) = \sum_{i=1}^N \lambda_i x_i^2,
  \end{equation*}
  where $\lambda_i > 0$ for all $i=1,\dots,N$. It holds that
  \begin{equation*}
    \ird e^{-\frac{Q(x)}{D}} x^\beta \d x
    =
    D^{\frac{N + |\beta|}{2}}
    \left( \prod_{i=1}^N \lambda_i^{-(1+\beta_i)/2} \right)
    M_\beta
  \end{equation*}
with $|\beta|=\sum_i \beta_i$.
\end{lem}

\begin{proof}
  Carrying out the change of variables $$y = (y_1,\dots,y_N) =
  \frac{1}{\sqrt{D}}(\sqrt{\lambda_1} x_1,\dots, \sqrt{\lambda_N}
  x_N)$$ directly yields the given expression.
\end{proof}

From the previous result one directly obtains the asymptotics of
integrals with any polynomial $g(x)$ instead of $x^\beta$. By a linear
change of coordinates it is then also simple to extend the result to
include any positive definite quadratic form in the exponential, but
we will not need that for our purposes. Extending the result to more
general functions in the exponential requires a more careful argument
that we give next. We begin by noticing that the only region
asymptotically contributing to the integral is concentrated around
$x=0$:
\begin{lem}
  \label{lem:Laplace-basic-2}
  Call $\delta := D^{1/3}$. Under the same conditions as in Lemma
  \ref{lem:Laplace-basic},
  \begin{equation*}
    \int_{|x| < \delta} e^{-\frac{Q(x)}{D}} x^\beta \d x
    =
    D^{\frac{N + |\beta|}{2}}
    \left( \prod_{i=1}^N \lambda_i^{-(1+\beta_i)/2} \right)
    M_\beta
    +
    \mathcal{O}\left( \exp\left(
      - \frac12 \lambda N D^{-1/3}
    \right)
 \right)
  \end{equation*}
  as $D \to 0$, where $\lambda := \min_{i=1,\dots,N} \lambda_i$. The
  constant implicit in the $\mathcal{O}$ notation depends continuously
  on the coefficients $\lambda_i$.
\end{lem}

\begin{proof}
  Due to Lemma \ref{lem:Laplace-basic} we just need to estimate the
  difference to the integral over all of $\R^N$:
  \begin{equation*}
    \int_{|x| \geq \delta} e^{-\frac{Q(x)}{D}} x^\beta \d x
    \leq
    e^{-\frac{\delta^2}{2D} \lambda N}
    \int_{|x| \geq \delta} e^{-\frac{Q(x)}{2D}} x^\beta \d x
    \leq
    \exp\left(
      - \frac12 \lambda N D^{-1/3}
    \right)
    \ird e^{-\frac{Q(x)}{2}} |x|^\beta \d x,
  \end{equation*}
  where we have assumed $D \leq 1$ since the statement concerns only
  the asymptotics as $D \to 0$. This gives an explicit bound of the
  remainder term.
\end{proof}

We now state the main result on Laplace's method that we use in this
paper:

\begin{thm}\label{bigkahuna}
  Let $P: \R^N \to \R$ be a polynomial function given by
  \begin{equation*}
    P(x) = a_0 + \sum_{i=1}^N \lambda_i x_i^2 + R(x) = a_0 + Q(x) + R(x),
  \end{equation*}
  where $a_0 \in \R$, $\lambda_i > 0$ for all $i=1,\dots,N$ and $R$
  contains only terms of degree three or higher. Assume that for some
  $\mu > 0$,
  \begin{equation}
    \label{eq:gm}
    P(x) - a_0 \geq \mu \min\{ |x|^2, 1\}
    \quad \text{ for all $x \in \R^N$.}
  \end{equation}
  Let $g: \R^N \to \R$ be a polynomial such that
  \begin{equation}
    \label{eq:finite}
    \ird e^{-\frac{P(x)}{D}} |g(x)| \d x < \infty
  \end{equation}
  for all $D \leq 1$. For $n \in \NN$ we have the expansion
  \begin{equation*}
    D^{-N/2} e^{\frac{a_0}{D}} \ird e^{-\frac{P(x)}{D}} g(x) \d x
    =
    \left( \prod_{k=1}^N \lambda_k^{-1/2} \right)
    M_0\, g(0)
    + \sum_{i=1}^n K_i D^i + \mathcal{O}(D^{n+1}), \quad \mbox{as } D\to 0\,,
  \end{equation*}
  where the numbers $K_i$ and the constant implicit in the
  $\mathcal{O}$ notation depend continuously only on $\mu$ and the
  coefficients of $P$ and $g$, and $M_0 := \ird e^{-|x|^2}$. In
  addition, $K_i$ depends only on the coefficients of $g$ of degree at
  most $2i$, and is equal to $0$ if $g$ has no such terms.
\end{thm}

\begin{remark}
  Condition \eqref{eq:gm} implies in particular that the global
  minimum of $P$ is attained at $x=0$. Moreover, it requires that the
  minimum be strict in a specific sense.
\end{remark}

\begin{proof}
  The constant term $a_0$ obviously gives the exponential factor
  $e^{-a_0/D}$, so me may assume that $a_0 = 0$. For $D \leq 1$ we
  choose $\delta := D^{1/3}$, and break the integration into the
  region inside the ball $B_\delta(0)$ and its complement.  We first
  show that the integration outside this ball is very small as
  $D \to 0$: using \eqref{eq:gm} and the inequality
  $\min(|x|^2,1)\geq \delta^2$ for $|x| \geq \delta=D^{1/3}$ and
  $D\leq 1$, then an argument very close to that in Lemma
  \ref{lem:Laplace-basic-2} shows that
  \begin{equation*}
    \int_{|x| > \delta}  e^{-\frac{P(x)}{D}} g(x) \d x
    \leq
    e^{-\frac{\mu \delta^2}{2D}}
    \int_{|x| > \delta}  e^{-\frac{P(x)}{2D}} |g(x)| \d x
    \leq
    e^{-\frac{\mu}{2} D^{-1/3}}
    \ird  e^{-\frac{P(x)}{2}} |g(x)| \d x
    =
    C_2     e^{-\frac{\mu}{2} D^{-1/3}}
  \end{equation*}
  valid for all $D \leq 1$. For the integral inside the ball of radius
  $\delta$, denote
  \begin{equation*}
    P(x) = Q(x) + R(x),
  \end{equation*}
  where $Q$ is the sum of all second-order terms of $P$ and $R$ is the
  sum of the remaining terms (of order greater than or equal to 3). We
  can expand $e^{-\frac{R(x)}{D}}$ to obtain, for $|x| \leq \delta$,
  \begin{equation}
    \label{eq:R-expansion}
    e^{-\frac{R(x)}{D}}
    =
    \sum_{i=0}^{2n+1} (-1)^i\frac{R(x)^i}{i! \,D^i}
    + \mathcal{O}\left( \frac{R(x)^{2n+2}}{D^{2n+2}} \right)
    =
    \sum_{i=0}^{2n+1} (-1)^i\frac{R(x)^i}{i! \,D^i}
    + \mathcal{O}\left( \frac{|x|^{6n+6}}{D^{2n+2}} \right),
  \end{equation}
  where the implicit constant depends only on a bound of $R$ by
  $|x|^3$ (which can be chosen as a continuous function of its coefficients).
  We then have, using Lemma \ref{lem:Laplace-basic-2} in order to
  estimate the remainder term,
  \begin{equation*}
    D^{-N/2} \int_{|x| \leq \delta}  e^{-\frac{P(x)}{D}} g(x) \d x
    =
    D^{-N/2} \int_{|x| \leq \delta}
    e^{-\frac{Q(x)}{D}}
    \left(
    \sum_{i=0}^{2n+1} (-1)^i\frac{R(x)^i}{i!\,D^i}
    \right)    g(x) \d x
    +
    \mathcal{O}\left( D^{n+1} \right).
  \end{equation*}
  We can use again Lemma \ref{lem:Laplace-basic-2} to estimate each
  term, since each of them is a quadratic exponential times a
  polynomial. Now, let us remember that $M_\beta=0$ in Lemma
  \ref{lem:Laplace-basic-2} whenever there is an index $k$ such that
  $\beta_k$ is odd. This suggests rewriting the polynomial in the
  integrand as
  $$
  \left(
    \sum_{i=0}^{2n+1} (-1)^i\frac{R(x)^i}{i!\,D^i}
  \right)    g(x) = \sum_{i=0}^{2n+1} \frac{r_e^i(x)}{D^{i}}
    + \sum_{i=0}^{2n+1} \frac{r_o^i(x)}{D^{i}}
  $$
  where the $r_e^i(x)$ are even polynomials (all their monomials are
  even) and $r_o^i(x)$ are odd polynomials. Therefore, Lemma
  \ref{lem:Laplace-basic-2} gives that
  \begin{equation*}
    D^{-N/2} \int_{|x| \leq \delta}  e^{-\frac{P(x)}{D}} g(x) \d x
    =
    D^{-N/2} \int_{|x| \leq \delta}
    e^{-\frac{Q(x)}{D}}
    \left(
    \sum_{i=0}^{2n+1} \frac{r_e^i(x)}{D^{i}}
    \right)  \d x
    +
    \mathcal{O}\left( D^{n+1} \right).
  \end{equation*}
  Now, we can identify the full expansion in powers of $D$. The only
  term that contains terms of the lowest order is the first one:
  \begin{equation*}
    D^{-N/2}
    \int_{|x| \leq \delta}
    e^{-\frac{Q(x)}{D}}
    g(x) \d x
    =
    \left( \prod_{i=1}^N \lambda_i^{-1/2} \right)
    M_0\, g(0)
    +
    \mathcal{O}(D).
  \end{equation*}
  The rest of the terms have order strictly higher than this, and it
  is seen from Lemma \ref{lem:Laplace-basic} that their coefficients
  are continuous functions of the coefficients of $P$ and $g$. We now
  observe that only integer powers of $D$ will appear in the expansion
  due to the evenness of the polynomials in the remainder and Lemma
  \ref{lem:Laplace-basic}, note that $|\beta|$ is even for all the
  monomials in the expansion. One also sees that if $g$ contains no
  terms of degree lower than or equal to $i$ then every term in the
  expansion is of order at least $D^{N/2} D^{i+1}$, and hence the
  coefficient $K_i$ is equal to 0.
\end{proof}

\subsection{$D\to 0$ Case: Multiple Solutions}
\label{S:anyDimension}

We will now show the existence of a curve of nonsymmetric stationary
states emanating from the stationary states for $D=0$ (which are the
measures $\delta_{u}$, for any $u \in \R^N$ with $|u| = 1$). Since
stationary states are determined by the roots of equation
\eqref{defH}, we are interested in the behavior of $H(u,D)$ as
$D \to 0$. The parameter $1/D$ appears inside the exponential, and the
asymptotics of integrals of this form is given by Laplace's method,
particularly by the statements given in the previous section.

Intuitively, we expect the stationary distribution to approach a Dirac
delta at the minimum of the polynomial $P_u(v)$ from eq. \eqref{poly}
as $D\to 0$ (this will be rigorously justified by Lemma
\ref{bigkahuna} as we will see below). Let us assume for the moment
that for $u > 0$ there is a unique minimum of $P_u(v)$ that is
achieved at $v = v^*(u)$ with $v_k^*(u)=0$, $k>1$ (we will also come
back to this point next to check that this assumption is
met). Therefore, if we want to find $u>0$ such that $H(u,0)=0$, we can
compute $H(u,0)$ formally at this point by substituting $f$ by a Dirac
delta at $v^*(u)$ in \eqref{defH}, leading to $H(u,0)=v_1^*(u)-u$, and
thus $u=v_1^*(u)$ at a point where $H(u,0)=0$. Now, since $u$ should
be equal to the point $v_1^*(u)$ at which the global minimum of $P_u(v)$ is
achieved, then $u=v_1^*(u)$ must satisfy $u^3=u$ according to the critical point condition
$\tfrac{\partial P_u}{\partial v_1}=\alpha v_1^3 + (1-\alpha) v_1 - u = 0$. Formally then the unique positive solution to $H(u,0) = 0$ is $u=1$. 

In this section, we will rigorously justify this, and follow a
perturbative argument to show that there is curve of solutions of
$H(u,D) = 0$ that converges to $(1,0)$ as $D \to 0$.

\paragraph{Global minima of $P_u$.}
Let us first find the minima of the polynomial $P_u(v)$. Its gradient
is
$$\nabla_v P_u(v) = \alpha |v|^2 v + (1-\alpha) v - u e_1.$$ The
critical points of $P_u(v)$ are thus characterized by
$\nabla_v P_u(v) \cdot e_k = 0$ for all $k=1,\dots,N$ (where
$\{e_1,\dots,e_N\}$ is the usual base of $\R^N$). That is,
\begin{align}
  \nabla_v P_u(v) \cdot e_k
  &=  v_k \left[ \alpha |v|^2 + (1-\alpha) \right] = 0
    \qquad (k = 2, \dots, N),
    \nonumber
  \\
  \nabla_v P_u(v) \cdot e_1
  &= v_1 \left[ \alpha |v|^2 + (1-\alpha)
               \right]-u = 0,
    \label{E:nabla1}
\end{align}
which means that either $v_k=0$ for all $k\neq 1$, or
$|v|^2=\frac{\alpha-1}{\alpha} = 1-\frac{1}{\alpha}$. For $u > 0$ we
cannot find critical points for which $|v|^2=1-\frac{1}{\alpha}$ due
to \eqref{E:nabla1}. Hence, in the case $u > 0$, critical points must
satisfy $v_k=0$ for all $k \neq 1$ and for all $\alpha>0$. Therefore,
in the case $u > 0$ all the critical points $v = (v_1,\dots,v_N)$
satisfy
\begin{equation}
  \label{eq:cpc}
  \alpha v_1^3 + (1-\alpha) v_1 - u = 0.
\end{equation}
The case $u=1$ can be explicitly solved since
$\alpha v_1^3 + (1-\alpha) v_1 - 1 = (\alpha v_1^2 + \alpha v_1
+1)(v_1-1)=0$,
so the roots are $v_1=1$ and
$v_1=\tfrac{-1 \pm \sqrt{\alpha - 4}}{2}$. It is simple to check that
$v_1 = 1$ is the unique global minimum of $P_1(v)$. In general, for
any $u > 0$ it can be seen that \eqref{eq:cpc} has one positive root
$v_1 = v_1^*(u)$, and the remaining roots are either complex or
negative, depending on the values of $\alpha$ and $u$ (this can be
checked by differentiating again in $v_1$). Since $u>0$, it is easy to check that
$P_u(-v_1,v_2,\dots,v_N) > P_u(v_1,v_2,\dots,v_N)$ for all
$v_1 > 0$, then the global minimum of $P_u(v)$ must be attained only at
$v = v^*(u) = (v_1^*(u),0,\dots,0)$. We have then proved the
following:
\begin{lem}[Global minimum of $P_u$]
  For $u > 0$ and $\alpha > 0$ the polynomial $P_u$ attains its global
  minimum only at
  \begin{equation*}
    v^*(u) = (v_1^*(u), 0, \dots, 0),
  \end{equation*}
  where $v_1^* = v_1^*(u)$ is a continuous function of $u > 0$, is positive
  for $u > 0$, and $v_1^*(1) = 1$.
\end{lem}
Of course, $v^*$ also depends on $\alpha$, but we omit this dependence
in the notation since $\alpha$ is fixed in all arguments below. Notice
that the continuity in $u$ of $v^*(u)$ is a consequence of the
continuity in $u$ of the unique positive root of \eqref{eq:cpc}.

\paragraph{A decomposition of $P_u$.}
For $u > 0$ we can write our polynomial $P_u(v)$ as
\begin{equation}\label{pol1}
P_u(v)=a_0(u)+Q_u(v)+R_u(v)
\end{equation}
with $a_0(u)=P_u(v^*(u))$ and 
\begin{equation}\label{pol2}
Q_u(v)= \langle \Lambda(v-v^*(u)),v-v^*(u) \rangle = \bar Q_u(v-v^*(u))\,,
\end{equation}
where $\Lambda$ is the Hessian matrix of $P_u(v)$ at the global
minimum $v^*(u)$. It can be calculated in terms of $v_1^*$ as
\begin{gather}
  \label{pol3}
  \Lambda=\operatorname{diag}(\lambda_1,\dots,\lambda_N) \quad \mbox{with }
  \lambda_1=\frac{1-\alpha}{2} + \frac32 \alpha v_1^*(u)^2
  \,\,\text{ and }\,\,
  \lambda_i = \frac{1-\alpha}{2} + \frac12\alpha v_1^*(u)^2, \, (i=2,\dots,N).
\end{gather}
Since $P_u(v)$ is of degree at most 4, the remainder $R_u(v)$ is of
the form
$$
R_u(v)=\sum_{|\beta|=3} a_\beta(u) (v-v^*(u))^\beta
+ \sum_{|\beta|=4} a_\beta(u) (v-v^*(u))^\beta,
$$
where the sum is over multiindices of the given order. In particular,
at $u=1$ we get
\begin{gather}
  \label{eq:a0Q1}
  a_0(1)=-\tfrac{\alpha+2}4\,,
  \qquad Q_1(v)= \tfrac12 |v-e_1|^2+\alpha
  (v_1-1)^2 = \bar Q_1(v-e_1),
  \\
  \label{eq:R1}
  R_1(v)=\alpha |v-e_1|^2
  (v_1-1) + \tfrac{\alpha}{4} |v-e_1|^4\,.
\end{gather}
\paragraph{Auxiliary functions.}  In order to make use of Theorem
\ref{bigkahuna} here and in later sections, let us define the
functions
$$
F_k(u,D) = \int_{\R^N} (v_1-v_1^*(u))^k
\exp\left\{-\frac{P_u(v)}{D}\right\} \d v,
$$
with $k=0,1,2$. Applying Theorem \ref{bigkahuna} to $F_k(u,D)$ we conclude that 
\begin{equation}\label{lap1}
F_0(u,D)= e^{-a_0(u)/D} D^{N/2} (c_0(u)+\mathcal{O}(D))\,,
\end{equation}
\begin{equation}\label{lap2}
F_1(u,D)= e^{-a_0(u)/D} D^{N/2} (c_1(u)D+\mathcal{O}(D^2))
\end{equation}
and
\begin{equation}\label{lap3}
F_2(u,D)= e^{-a_0(u)/D} D^{N/2} (c_2(u)D +\mathcal{O}(D^2))\,,
\end{equation}
for $u>0$ as $D\to 0$, where $c_k(u)$, $k=0,1,2$, are continuous
functions of $u$, and the constants implicit in the $\mathcal{O}$
notation are uniform in a neighborhood of $u = 1$ (one can check that
all conditions in Theorem \ref{bigkahuna} hold uniformly in a
neighborhood of $u=1$). The explicit expression of the first term in
the expansion in Theorem \ref{bigkahuna} gives
\begin{equation}
  \label{eq:c01}
  c_0(1)=\int_{\R^N} e^{-\bar Q_1(z)}\dz = (2\pi)^{N/2} \sqrt{\frac1{1+2\alpha}}\,,
\end{equation}
and thus by continuity we have $c_0(u) \neq 0$ in a neighborhood of
$u=1$. For reference below, we take $\epsilon_0 > 0$ such that
\eqref{lap1}-\eqref{lap3} hold for $|u-1| < \epsilon_0$.  Analogously,
we can use the expansion in Theorem \ref{bigkahuna} to obtain the
first order term of the functions $F_1(u,D)$ and $F_2(u,D)$ at $u=1$
to get
\begin{equation}
  \label{eq:c11}
  c_1(1)= - \alpha \int_{\R^N} z_1^2 |z|^2 e^{-\bar Q_1(z)}\dz
  =
  -\alpha (2\pi)^{N/2} (1+2\alpha)^{-5/2}
  \left(N + 2 + 2 (N-1) \alpha
  \right) < 0
\end{equation}
and
$$
c_2(1)=\int_{\R^N} z_1^2 e^{-\bar Q_1(z)}\dz = (2\pi)^{N/2} \left(\frac1{1+2\alpha}\right)^{3/2}\,.
$$
(See Appendix \ref{sec:appendix} for the explicit calculations leading
to this.)

\paragraph{Continuity of $H$ as $D \to 0$.}
The function $H(u,D)$ is smooth with respect to $u$ and $D$ for all
$u>0$ and $D>0$, as can be seen by standard arguments. Let us show
that $H(u,D)$ has a limit as $D \to 0$ (which will enable us to define
it by continuity at $D=0$). It is easy to verify the following
formulas that relate $H$ to $F_0$ and $F_1$:
\begin{equation}\label{for1}
  Z= F_0(u,D)
  \qquad \mbox{and} \qquad
  H(u,D)=\frac{F_1(u,D)}{F_0(u,D)}+(v_1^*(u)-u)\,.
\end{equation}
We deduce from \eqref{for1} taking into account \eqref{lap1} and
\eqref{lap2} that
$$
\lim_{D\to 0} H(u,D)= v_1^*(u)-u \qquad \mbox{since}\quad \frac{F_1(u,D)}{F_0(u,D)}=\frac{c_1(u)}{c_0(u)} D+\mathcal{O}(D^2)\,,
$$
for $|u-1| < \epsilon_0$. As a consequence, by defining
\begin{equation}
  \label{extendH}
  H(u,0)=v_1^*(u)-u
\end{equation}
the function $H(u,D)$ is continuous in
$(u-\epsilon_0, u+\epsilon_0) \times [0,+\infty)$.

\paragraph{Differentiability in $u$.}

It is straightforward to check that
\begin{equation}\label{for2}
\frac{\partial H}{\partial u}(u,D)=\frac1D \frac{F_0(u,D)F_2(u,D)-F_1(u,D)^2}{F_0(u,D)^2}-1
\end{equation}
We proceed as before: applying \eqref{lap1}-\eqref{lap3} in
\eqref{for2} we obtain
$$
\lim_{D\to 0} \frac{\partial H}{\partial u}(u,D)= \frac{c_2(u)}{c_0(u)}-1 \qquad \mbox{since}\quad \frac{F_0(u,D)F_2(u,D)-F_1(u,D)^2}{F_0(u,D)^2}=\frac{c_2(u)}{c_0(u)} D+\mathcal{O}(D^2)\,,
$$
for $|u-1|<\varepsilon_0$. This shows the function $H(u,D)$ (extended
to $D=0$ as in \eqref{extendH}) is
differentiable with respect to $u$ in a neighborhood of $(1,0)$. It
is simple to check that
\begin{equation}\label{finally}
\frac{\partial H}{\partial u}(1,0)= \frac{c_2(1)}{c_0(1)}-1 = \frac1{1+2\alpha}-1=-\frac{2\alpha}{1+2\alpha}\neq 0\,.
\end{equation}
This comes again from the explicit computation of the first term in
the expansion of Theorem \ref{bigkahuna} applied to $F_2(u,D)$ which is
given by the second moment centered at $v_1^*(1)=1$ of $Q_1(v)$.

\paragraph{Differentiability in $D$.}
In a similar way we can write
\begin{align}
\frac{\partial H}{\partial D}(u,D)=\frac1{D^2\,F_0(u,D)} &\left( \int_{\R^N} (v_1-u) P_u(v) \exp\left\{-\frac{P_u(v)}{D}\right\} dv \right.\nonumber\\
&\left.- H(u,D) \int_{\R^N} P_u(v) \exp\left\{-\frac{P_u(v)}{D}\right\} dv \right)\,.
\label{for31}
\end{align}
Inserting \eqref{for1} into \eqref{for31}, this is equivalently written
as
\begin{align}
\frac{\partial H}{\partial D}(u,D)=\frac1{D^2\,F_0(u,D)^2} &\left( F_0(u,D) \int_{\R^N} (v_1-v_1^*(u)) P_u(v) \exp\left\{-\frac{P_u(v)}{D}\right\} dv \right.\nonumber\\
&\left.- F_1(u,D) \int_{\R^N} P_u(v) \exp\left\{-\frac{P_u(v)}{D}\right\} dv \right),
\label{for32}
\end{align}
and with the decomposition \eqref{pol1} we get the expression
\begin{align}
\frac{\partial H}{\partial D}(u,D)=\frac1{D^2\,F_0(u,D)^2} &\left( F_0(u,D) \int_{\R^N} (v_1-v_1^*(u)) (Q_u(v)+R_u(v)) \exp\left\{-\frac{P_u(v)}{D}\right\} dv \right.\nonumber\\
&\left.- F_1(u,D) \int_{\R^N} (Q_u(v)+R_u(v)) \exp\left\{-\frac{P_u(v)}{D}\right\} dv \right)\,.
\label{for33}
\end{align}
Applying Theorem \ref{bigkahuna} to the two integrals in \eqref{for33}
to obtain
$$
\int_{\R^N} (v_1-v_1^*(u)) (Q_u(v)+R_u(v)) \exp\left\{-\frac{P_u(v)}{D}\right\} dv = e^{-a_0(u)/D} D^{N/2} (k_1(u) D^2+\mathcal{O}(D^3))
$$
and
$$
\int_{\R^N} (Q_u(v)+R_u(v)) \exp\left\{-\frac{P_u(v)}{D}\right\} dv = e^{-a_0(u)/D} D^{N/2} (k_2(u) D+\mathcal{O}(D^2))
$$
as $D\to 0$ where $k_1(u)$ and $k_2(u)$ are continuous functions for
$|u-1|<\varepsilon_0$. In fact, using the expansion in Theorem
\ref{bigkahuna} we obtain
\begin{equation}
  \label{eq:k11}
 k_1(1)=\alpha\int_{\R^N} z_1^2 |z|^2 \left(1-\bar Q_1(z)\right) e^{-\bar Q_1(z)}\dz
\end{equation}
and
$$
 k_2(1)=\int_{\R^N} \bar Q_1(z) e^{-\bar Q_1(z)}\dz
          = \frac{(2 \pi)^{N/2} N}{2 \sqrt{1+2\alpha}} \,,
$$
whose expressions are derived in Appendix \ref{sec:appendix}.
These expressions, together with the formulas \eqref{lap1} and
\eqref{lap2}, result in
\begin{equation}
  \label{eq:pHpD}
  \lim_{D\to 0} \frac{\partial H}{\partial D}(u,D)= \frac{c_0(u)
    k_1(u)-c_1(u) k_2(u)}{c_0(u)^2}
\end{equation}
for $|u-1|<\varepsilon_0$. This shows that the function $H(u,D)$ is
differentiable from the right with respect to $D$ in a neighborhood
of $(1,0)$. Moreover, we find that
$$
\frac{\partial H}{\partial D}(1,0) =  -\frac{\alpha}{(1+2\alpha)^2}
    \left(
           3+(N-1) (1 + 2 \alpha)
  \right)<0\,,
$$
for all $\alpha >0$, according to the values in Appendix \ref{sec:appendix}.

\paragraph{Existence of a curve of solutions close to $(u,D) =
  (1,0)$.}
Summarizing, we have proved that the function $H(u,D)$ is a
differentiable ($\mathcal{C}^1$) function in a neighborhood of
$(1,0)$, such that $H(1,0)=0$. Equation \eqref{finally} implies that $\tfrac{\partial H}{\partial u}(1,0)\neq 0$ which allows us to
apply the implicit function theorem, implying that there exists a
curve $u=u^*(D)$ defined for $D$ small enough such that
$H(u^*(D),D)=0$. This shows the existence of a curve of non-symmetric
stationary states emanating from the Dirac delta at $v=e_1$ for $D=0$.


\section{Numerical Results}\label{S:numerics}

In this section we numerically validate the results of the previous
section, finding the bifurcation curves and numerically showing the
stability of the stationary solutions.  In addition to demonstrating
the analytical results related to the parameter $D$, we explore the
effect of the value of the parameter $\alpha$ on the critical noise
threshold and the effect of both parameters on the stationary velocity
profile.  To emphasize the dependence of $H(u,D)$ on the value of the
parameter $\alpha$ we will use the notation $H_\alpha(u,D)$ throughout
this section.  By examining the roots of $H_\alpha(u,D)$, we
numerically validate in both one and two dimensions the fact that
there is more than one stationary state for small magnitudes of noise,
and that there is only one stationary solution for large noise.  Using
both $H_\alpha(u,D)$ and $\frac{\partial H_\alpha}{\partial u}(0,D)$,
we show where in $\alpha$-$D$ parameter space the transition from more
than one to one stationary solutions occurs.  We then numerically
explore the $\alpha \to \infty$ limit.

To examine further properties of the system we use a Monte Carlo-like
particle method to approximate the steady states and the transient
behavior of the system, employing the Euler-Maruyama method to
numerically solve the SDEs, see \cite{KloedenPlaten} for instance.
With this framework, we are able to numerically validate the stability
of the nonzero stationary solution when it exists, giving evidence
that this bifurcation is indeed a phase transition.  Using a large
ensemble of independent runs, we also track the temporal evolution of
both the average velocity and the free energy $\mathcal{F}$ defined in
Section \ref{S:homogeneousProblem}.

In order to efficiently compute the stationary states and the
bifurcation diagram in two dimensions, we use radial coordinates. In
fact, we can rewrite the function $H_\alpha(u,D)$ in radial
coordinates in any dimension as follows. Defining
$$
E_D(r)= \exp\left( \tfrac{\alpha}{D} (\tfrac{r^2}{2} - \tfrac{r^4}{4}) - \tfrac{r^2}{2D}) \right)\,
$$
we can reexpress $H_\alpha(u,D)$ in coordinates  in terms of the angle with respect to the first axis as:
\begin{align*}
H_\alpha(u,D) &= \frac{1}{Z} \int_0^\infty r^{N-1}E_D(r)\int_{S^{N-1}} (r \omega_1 - u)  \exp \left( \tfrac{u r\omega_1}{D} \right) \omega \,d\theta \,dr\\
&=\frac{1}{Z} \int_0^\infty r^{N-1} E_D(r) \int_{0}^{\pi} \int_{S^{N-2}} (r \cos \theta - u)  \exp\left( \tfrac{u r\cos \theta}{D}\right) d\tilde\omega \sin^{N-2}\theta \,d\theta \,dr\\
&= \frac{\int_0^\infty r^{N} E_D(r) \int_{0}^{\pi} \cos \theta \exp\left( \tfrac{u r\cos \theta}{D}\right) \sin^{N-2}\theta \,d\theta \,dr}{\int_0^\infty r^{N-1} E_D(r) \int_{0}^{\pi} \exp\left( \tfrac{u r\cos \theta}{D}\right) \sin^{N-2}\theta \,d\theta \,dr}-u.
\end{align*}
Let us use the notation
\begin{equation}\label{bessel1}
\mathcal{I}_n^N(z) = \int_{0}^{\pi} \cos^n(\theta) \exp(z \cos(\theta)) \sin^{N-2}\theta \,d\theta.
\end{equation}
We further reduce to:
\begin{align*}
H_\alpha(u,D) &= \frac{\int_0^\infty r^N  E_D(r)  \mathcal{I}_1^N \left( \frac{ur}{D}\right) dr}{\int_0^\infty r^{N-1}  E_D(r) \mathcal{I}_0^N \left( \frac{ur}{D} \right) dr} - u.
\end{align*}
Formula \eqref{bessel1} in the two-dimensional case leads to an
expression in terms of modified Bessel functions of the first kind and
in the three-dimensional case they can be obtained explicitly in terms
of hyperbolic sine and cosine functions. This is exploited in the
two-dimensional numerics of subsections \ref{S:bifurcations} and
\ref{S:roleOfAlpha}.  Note that these expressions do not simplify the
analytical discussion in Section 2 of the behavior of $H_\alpha(u,D)$
for large and small noise $D$.

\subsection{Bifurcations}\label{S:bifurcations}

We are able to numerically observe the bifurcations predicted by the
analysis in the previous section. In Figures \ref{F:dHduOverlap1D} and
\ref{F:dHduOverlap2D}, we show in one and two dimensions,
respectively, the root $u(D)$ of $H_\alpha(u,D)$ plotted as a function
of $D$ in solid lines. The curves were determined using a root-finding
function on $H_\alpha(u,D)$ for each fixed value of $D$. We also plot
$\frac{\partial H_\alpha}{\partial u}(u(D),D)$ as dotted lines for
varying values of $\alpha$. Here, the derivative is computed by
substituting the root $u(D)$ into the formula \eqref{for2} of
$\frac{\partial H_\alpha}{\partial u}(u,D)$, given in subsection 2.3.

Let us define the critical value of the noise $D_c$ as the noise at
which $u(D)$ attains zero for the first time. From Figures
\ref{F:dHduOverlap1D} and \ref{F:dHduOverlap2D}, it is clear that in
both one dimension and two dimensions,
$\frac{\partial H_\alpha}{\partial u}(u(D),D)$ is equal to zero at
$D_c$. This tells us that the bifurcation branch intersects the zero
branch perpendicularly, since
$$
\frac{\partial u}{\partial D}=
-\frac{\partial H_\alpha}{\partial D}\left(\frac{\partial H_\alpha}{\partial u}\right)^{-1}\,.
$$
This is unsurprising, given the form of $H_\alpha(u,D)$ shown in
figure \ref{F:Hofu}: the slope of $H_\alpha$ evaluated at the nonzero
stationary solution should be negative, becoming increasingly shallow
until it becomes zero. We also note that the bifurcation diagram is
decreasing for $D$ small as indicated by the formulas found in
previous section, since
$$
\frac{\partial u}{\partial D} (0)
= -\frac{\partial H_\alpha}{\partial D}(1,0)\left(\frac{\partial
    H_\alpha}{\partial u}(1,0)\right)^{-1}
= -\frac{3+(N-1) (1 + 2 \alpha)}{2(1+2\alpha)}<0\,.
$$

\begin{figure}
\begin{center}
\includegraphics[width=16cm]{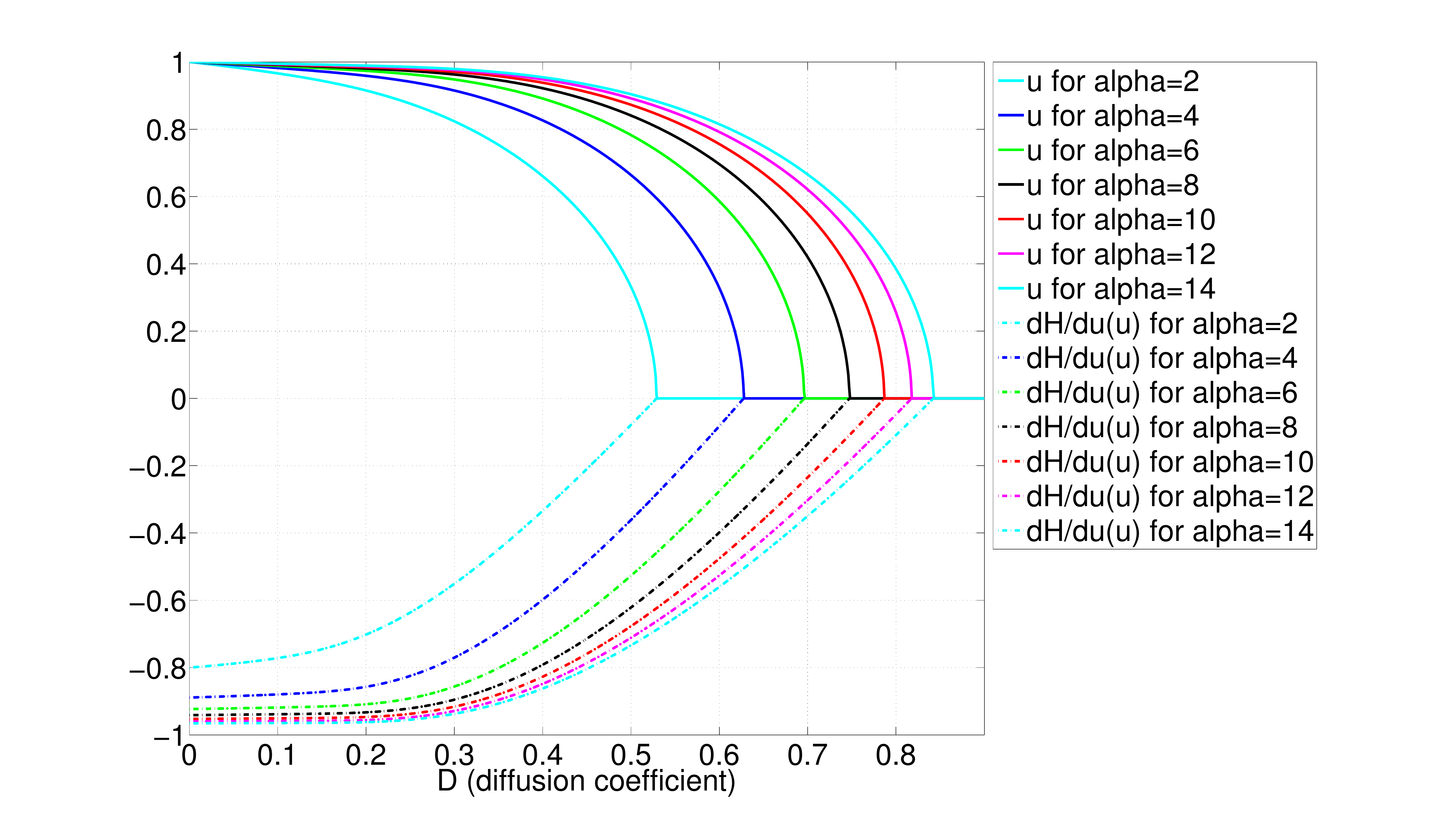}
\caption{One Dimension: Here, the solid lines are the roots $u(D)$ of
  $H_\alpha(u,D)$ in one dimension plotted against the diffusion
  coefficient $D$. The dotted lines are
  $\frac{\partial H_\alpha}{\partial u}(u(D),D)$, plotted as a
  function of $D$. The values of $\alpha$ are uniformly spaced with
  intervals of $2$ between $2$ and $14$.  This figure validates that
  the nonzero stationary solution disappears when
  $\frac{\partial H_\alpha}{\partial u}(u(D),D)$ reaches zero.}
\label{F:dHduOverlap1D}
\end{center}
\end{figure}

\begin{figure}
\begin{center}
\includegraphics[width=16cm]{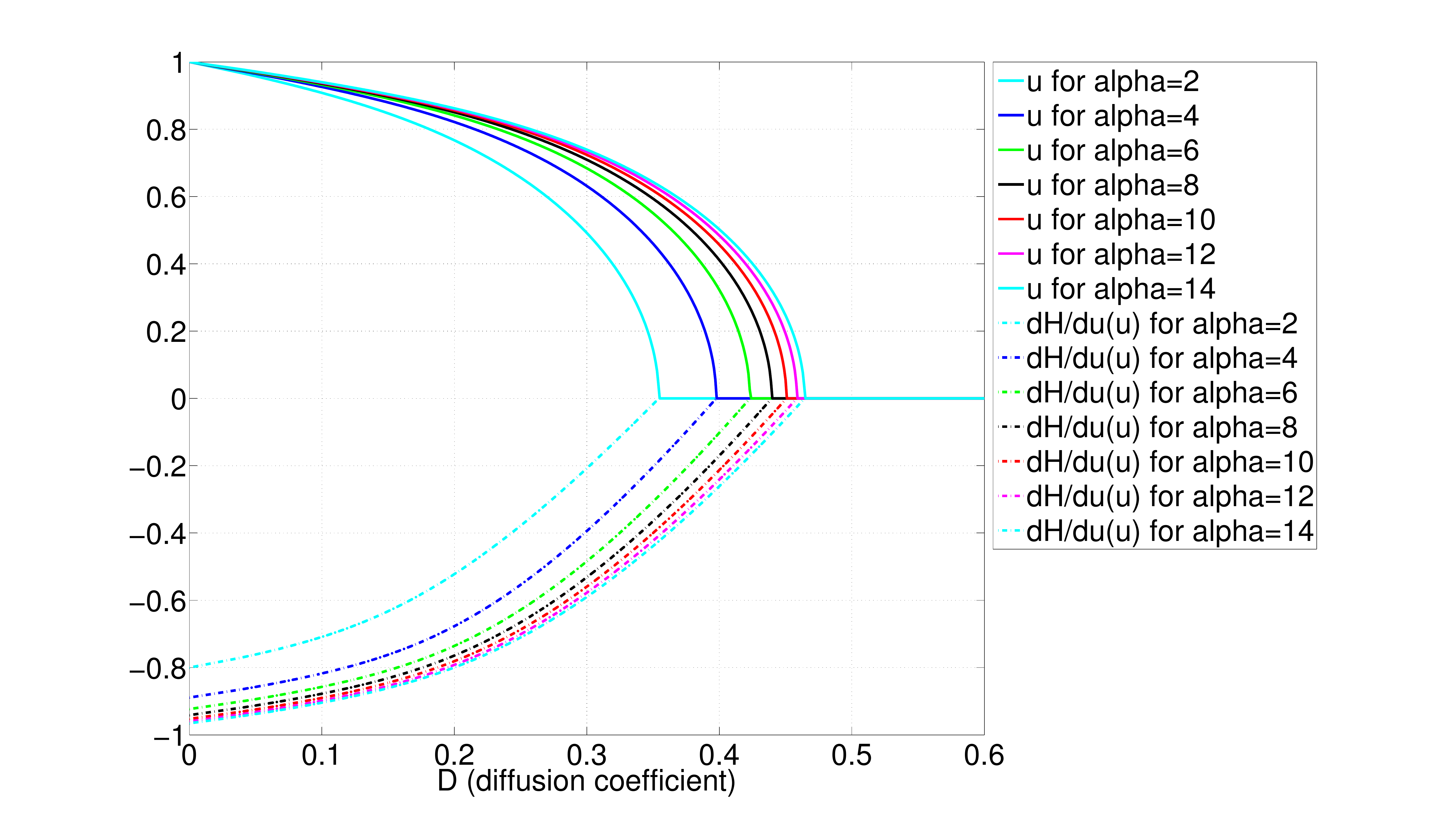}
\end{center}
\caption{Two Dimensions: Roots $u(D)$ of $H_\alpha(u,D)$ shown in
  solid lines and $\frac{\partial H_\alpha}{\partial u}(u(D),D)$ shown
  in dotted lines for equally spaced $\alpha$ varying from $2$ to $14$
  in two dimensions.  As can be seen, the roots of both $u(D)$ and
  $H_\alpha(u(D),D)$ coincide at the critical $D$.  Note also that
  this critical value is less than the critical value for the
  one-dimensional case seen in the previous figure.}
\label{F:dHduOverlap2D}
\end{figure}


\subsection{The role of $\alpha$} \label{S:roleOfAlpha}

According to the formulas \eqref{bessel1}, in two dimensions,
$I_0^2(0) = \pi$, $I_1^2(0) = 0$, and $I_2^2(0) = \frac{\pi}{2}$, so
\begin{align*}
F_0(0) = \pi \int_0^\infty r E_D(r) \,dr\qquad \mbox{and}\qquad
F_2(0) = \frac{\pi}{2} \int_0^\infty r^3 E_D(r) \,dr\,.
\end{align*}
Thus, in two dimensions,
\begin{align*}
  \frac{\partial H_\alpha}{\partial u}(0,D)
  &= \frac{1}{D} \frac{F_2(0)}{F_0(0)} - 1
    = \frac{2}{D} \frac{\int_0^\infty r \exp\left( \tfrac{\alpha}{D}
    (\tfrac{r^2}{2} - \tfrac{r^4}{4}) - \tfrac{r^2}{2D}) \right) \,dr}
    {\int_0^\infty r^3 \exp\left( \tfrac{\alpha}{D} (\tfrac{r^2}{2} - \tfrac{r^4}{4}) - \tfrac{r^2}{2D}) \right) \,dr} - 1.
\end{align*}
This calculation highlights the dependence of the bifurcation curve on
the parameter $\alpha$. We demonstrate this numerically in both one
and two dimensions in figures \ref{F:dHduOverlap1D} and
\ref{F:dHduOverlap2D}, respectively, where we can observe how the
bifurcation curves change as we vary $\alpha$.

In figure \ref{F:dHdu}, we numerically determine where in $\alpha$-$D$
parameter space this bifurcation occurs in one dimension.  This
bifurcation diagram was found both analytically and numerically by
Tugaut in \cite[Subsection 4.1]{Tugaut2013Phase}. Here, we sample the
parameter space, with $\alpha$ along the vertical axis and $D$ along
the horizontal one, and plot a blue dot when the point has more than
one stationary solution and a red dot where it has only one.  The
black line was drawn using the continuation method to find the root of
$\frac{\partial H_\alpha}{\partial u}(0,D)$.  It is clear from the
figure that for $D$ sufficiently large, we could consider $\alpha$ to
be the bifurcation parameter, identifying a critical value $\alpha_c$
at which a phase transition occurs. Unlike the case of increasing $D$,
this bifurcation is from one stationary solution to more than one as
$\alpha$ increases, and making this phase transition explicit remains
to be explored. We finally mention that we observe that the critical
value $D_c$ has a limit as $\alpha\to 0$, this fact was already
studied analytically in \cite[Subsection 4.1]{Tugaut2013Phase} showing
that its limit value is $1/3$ matching with our simulations.

\begin{figure}
\begin{center}
\includegraphics[width=10cm]{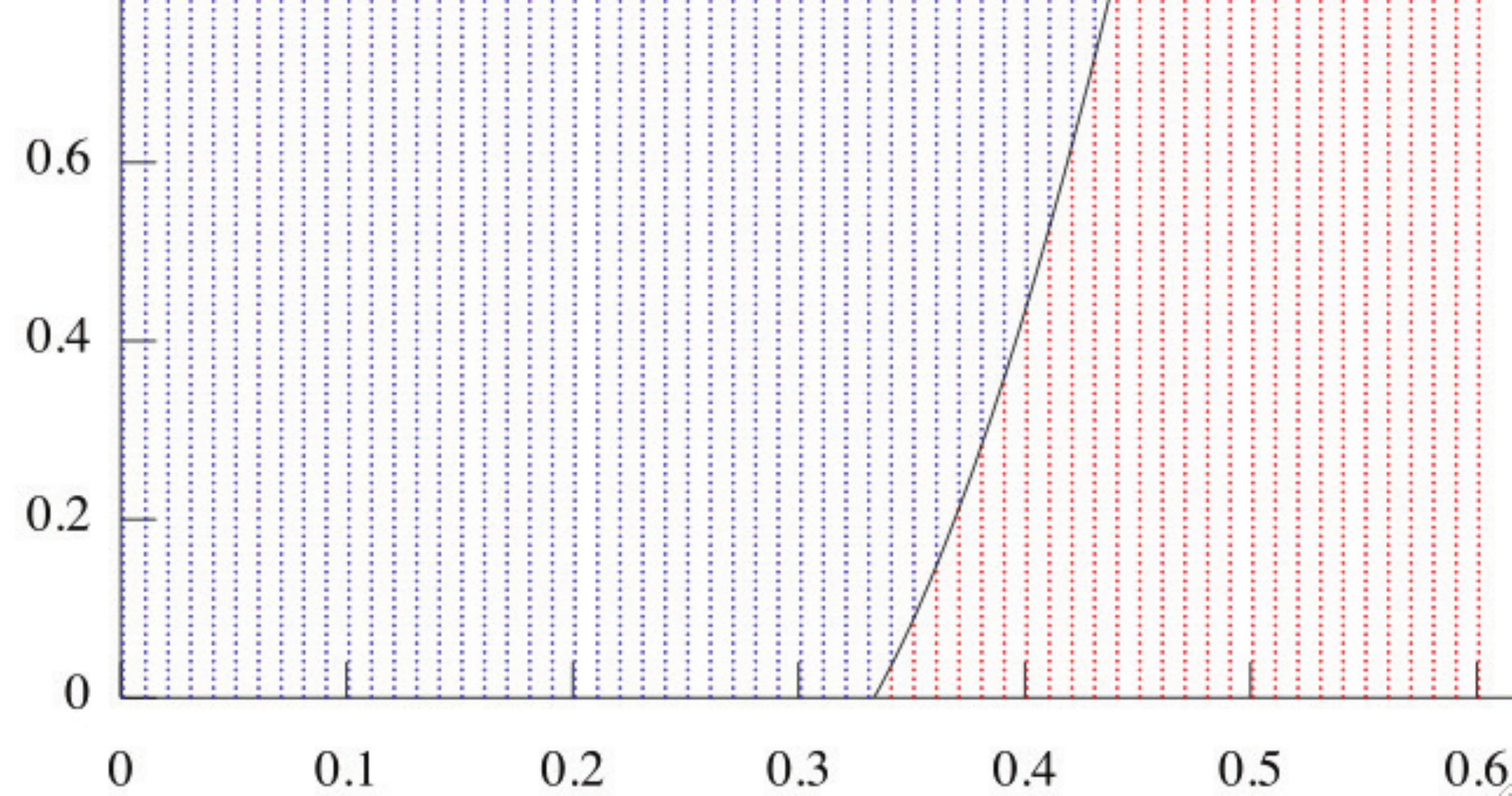}
\end{center}
\caption{This figure shows where in $\alpha$-$D$ space we can expect
  the bifurcation to occur.  The vertical axis is $\alpha$ and the
  horizontal axis is $D$.  We plot points in parameter space where
  there exists a nonzero stationary state in red, and points where we
  find only the zero stationary state in blue.  The line of
  demarcation between the two regions is created using a continuation
  method to find the root of
  $\frac{\partial{H_\alpha}}{\partial u}(0,D)$.  Taken together, these
  roots, which determine for which $D$ the slope of $H_\alpha$ at zero
  changes from positive to negative, define the line.}
\label{F:dHdu}
\end{figure}

\begin{figure}[ht]
\begin{center}
\includegraphics[width=11cm]{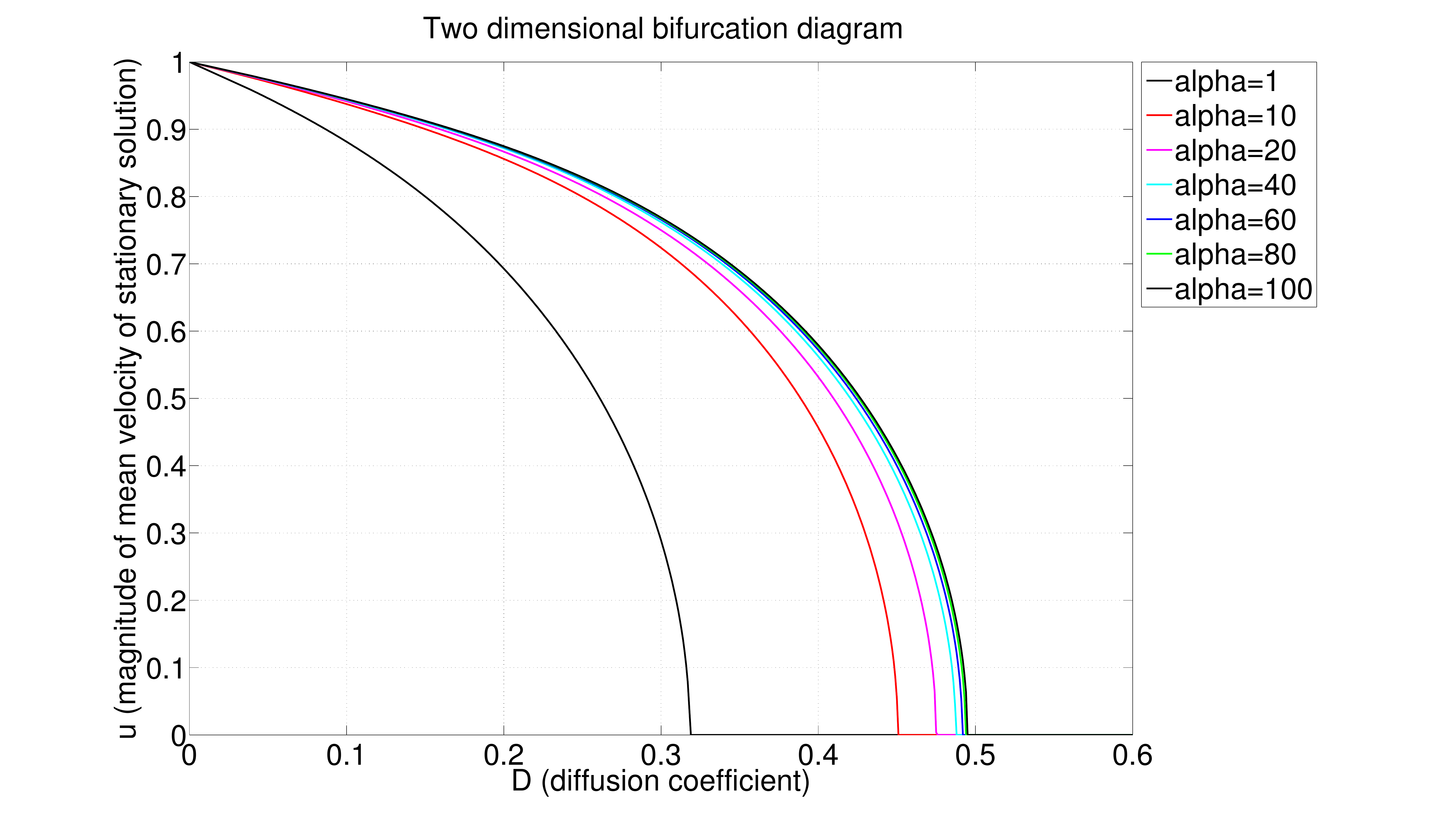}
\end{center}
\caption{Two-Dimensional Bifurcation Diagram: Here, the
  two-dimensional bifurcation curves are plotted for increasing values
  of $\alpha$. This figure indicate that there is a limiting
  bifurcation curve, and hence a limiting critical noise value
  approaching 1/2, as $\alpha \to \infty$.}
\label{F:bifurcation2D}
\end{figure}

It is interesting to note that the changes in the bifurcation curve
lessen as $\alpha$ increases, see figure \ref{F:bifurcation2D} in the
two-dimensional case.  This indicates that the bifurcation curves are
approaching a limiting function as $\alpha \to \infty$. Letting
$\alpha\to \infty$ means that the cruise speed term dominates the
behavior of the system; intuitively speaking, as $\alpha \to \infty$,
we recover a norm constraint in the velocity for particles. This
intuition was in part rigorously proved by Bostan and Carrillo in
\cite{bostan2013asymptotic}, where they show that the kinetic equation
\eqref{E:main} limits to the continuum Vicsek model in
\cite{dfl2013}. Here, we numerically conclude that a limiting phase
transition curve does indeed exist in two dimensions and it
qualitatively matches the one obtained in
\cite{frouvelle2012dynamics,dfl2013}. In fact, the critical noise
value $D_c$ is converging towards the critical noise value $1/2$ for
the Vicsek model obtained in \cite[Proposition
3.3]{frouvelle2012dynamics}.

\subsection{Stability and Phase Transition}\label{S:stability}

In order to numerically explore the stability of the stationary
solutions, we approximate the solutions to \eqref{E:main} by a Monte
Carlo-like method using the Euler-Maruyama numerical scheme.  We use
$10000$ particles per run with a timestep of $0.01$ and evaluate the
average velocity at time $6000$. This is enough for stabilization in
time of the solutions except for noise values close to the critical
noise parameter $D_c$, which are well known to take longer to
converge. In the particle simulations, we initialize the particles
with velocities sampled from the Gaussian $\mathcal{N}(0.5,1)$ in
order to investigate which of the stationary solutions is stable.  In
figure \ref{F:stability}, we plot this final average velocity over an
ensemble of ten runs on top of three of the bifurcation curves studied
in Figure \ref{F:dHduOverlap1D}.  From the figure, it is clear that
the particle simulations, initialized away from either stationary
state, approach the nonzero stationary solution while such a solution
exists.  As expected, this indicates that, when it exists, the nonzero
stationary solution is stable.  This demonstrates that the zero
stationary solution is unstable before the critical value of the noise
and stable afterward, indicating the bifurcation we observe is indeed
a phase transition.

\begin{figure}[ht]
\begin{center}
\includegraphics[width=10cm]{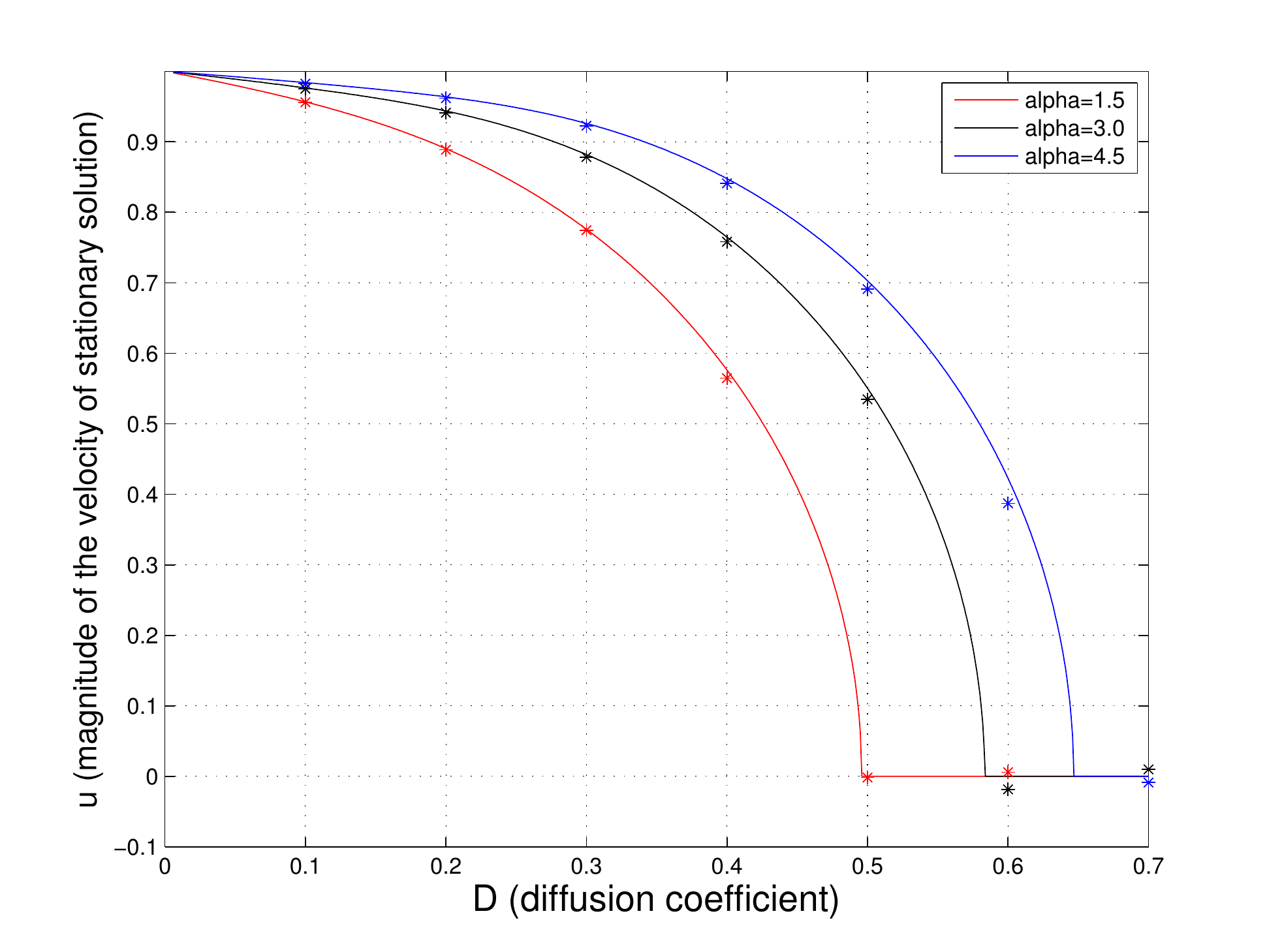}
\end{center}
\caption{One Dimension: This plot demonstrates the stability of the
  nonzero stationary solution for $\alpha=1.5, 3.0, \mbox{ and } 4.5$.
  Here, the average velocity over an ensemble of ten runs of the
  microscopic model is plotted (asterisks) over bifurcation curves
  (solid line) similar to those in figure \ref{F:dHduOverlap1D}. The
  particles were initialized with velocity sampled from
  $\mathcal{N}(0.5,1)$.  It is clear from this plot that the average
  velocity of the microscopic runs agrees with the nonzero stationary
  solution, indicating that this is the stable stationary state as
  long as it exists. }
\label{F:stability}
\end{figure}

\subsection{Stationary Solutions}

One validation of the efficacy of the numerical method with the
particles is whether we are able to recover the stationary solutions
for different values of $D$.  In Figure \ref{F:histAgainstTruth}, the
dots show the final histograms at time $500$ of an ensemble of $100$
runs in one dimension with $\alpha = 2$ and varying $D$.  In solid
lines, we plot the solution $f_{u(D)}$ given by equation
\eqref{stateq}, taking $u(D)$ to be the values computed in subsection
\ref{S:bifurcations}. We observe an impressive match between the
histograms and the analytical distribution, further validating the
Monte Carlo-like approach to the solutions.

\begin{figure}
\begin{center}
\includegraphics[width=16cm]{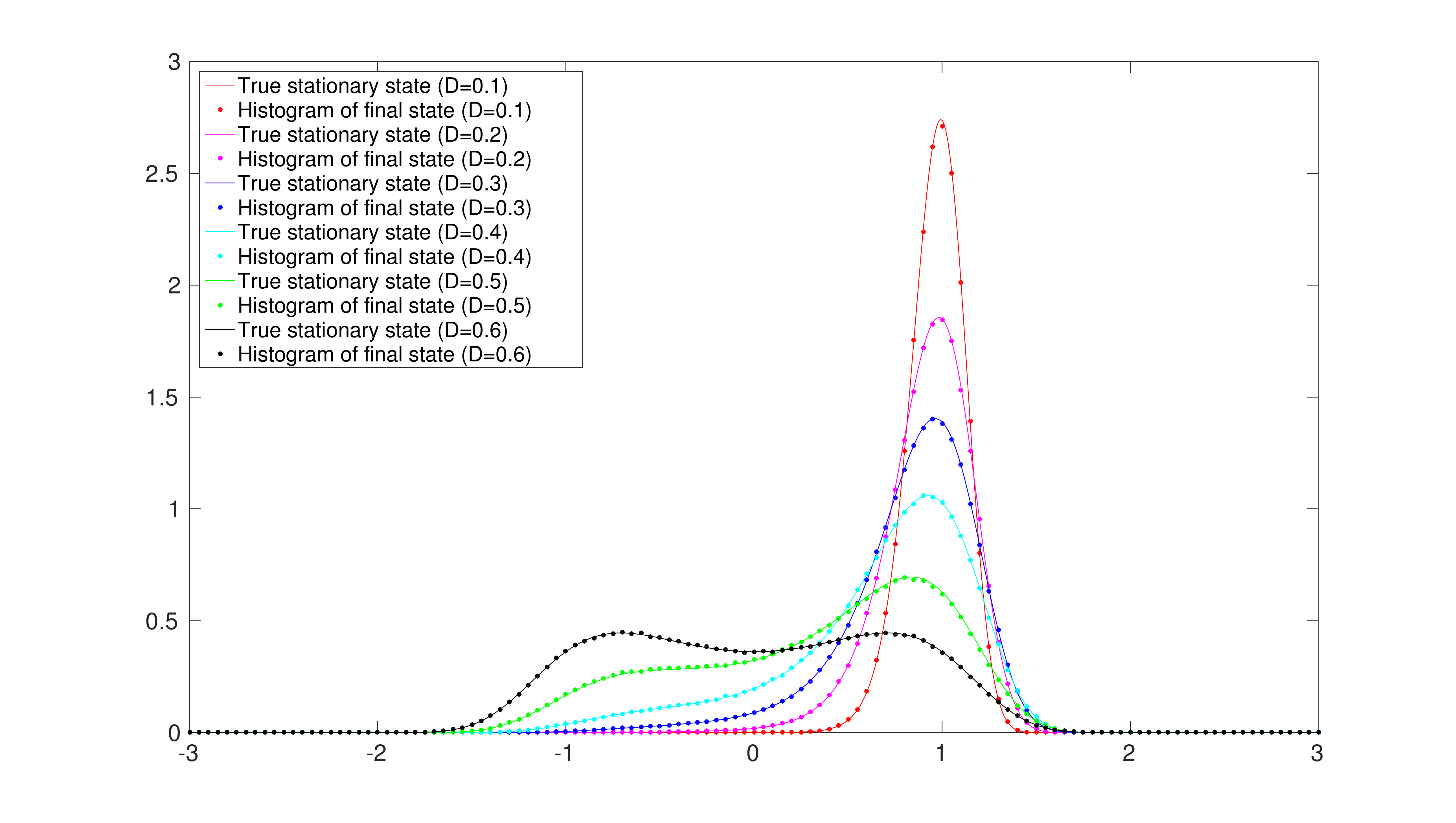}
\end{center}
\caption{One Dimension: points show the final
  velocity profiles at time $500$ from our particle simulation with
  $\alpha = 2$.  The solid lines are the corresponding stationary
  solutions $f_{u(D)}$, computed by substituting the stationary
  average velocity $u(D)$ from the bifurcation diagram shown in Figure
  \ref{F:dHduOverlap1D} into equation \eqref{stateq}.}
\label{F:histAgainstTruth}
\end{figure}

\begin{figure}
 \begin{center}
 \includegraphics[width=8cm]{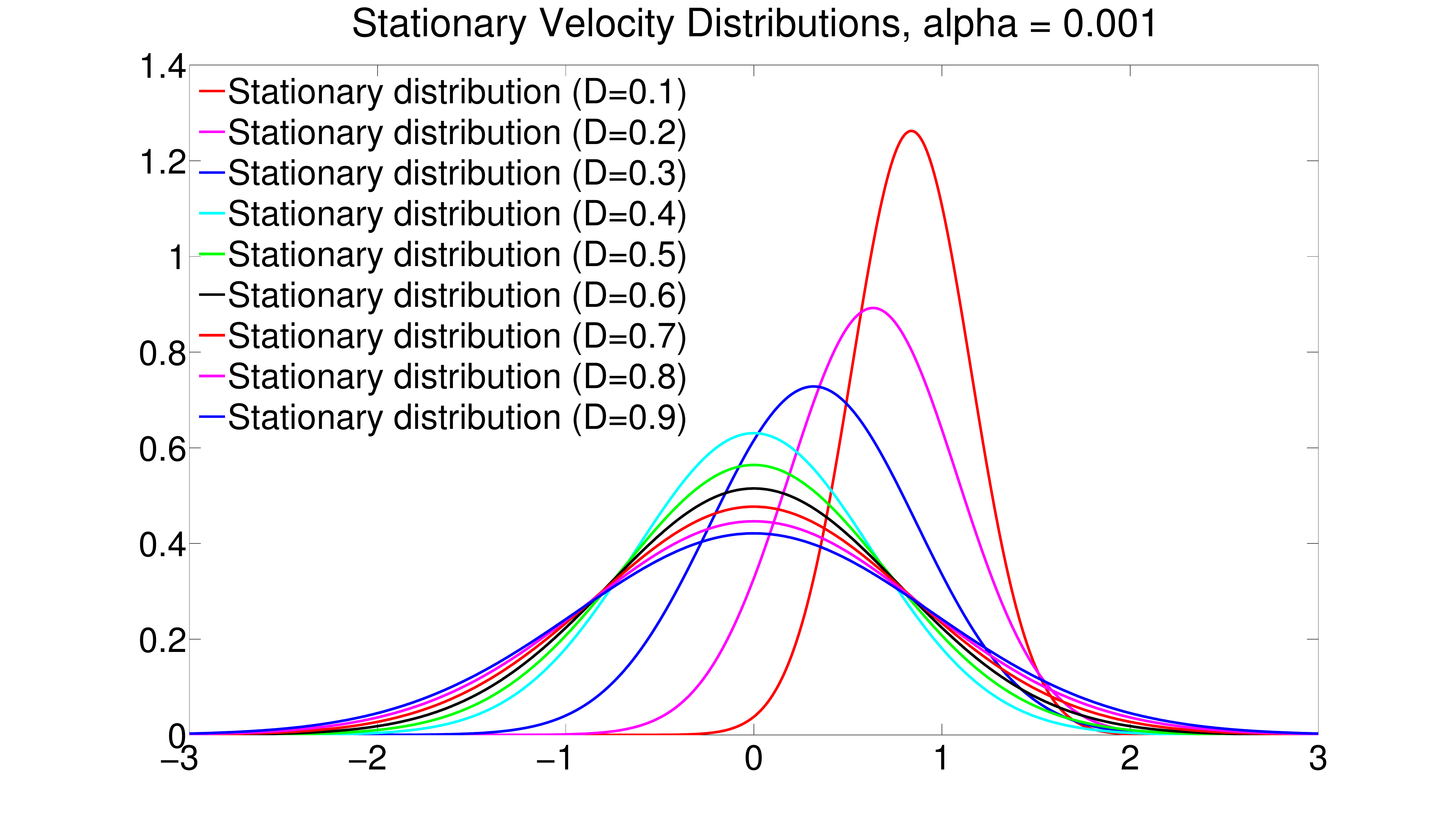}
 \includegraphics[width=8cm]{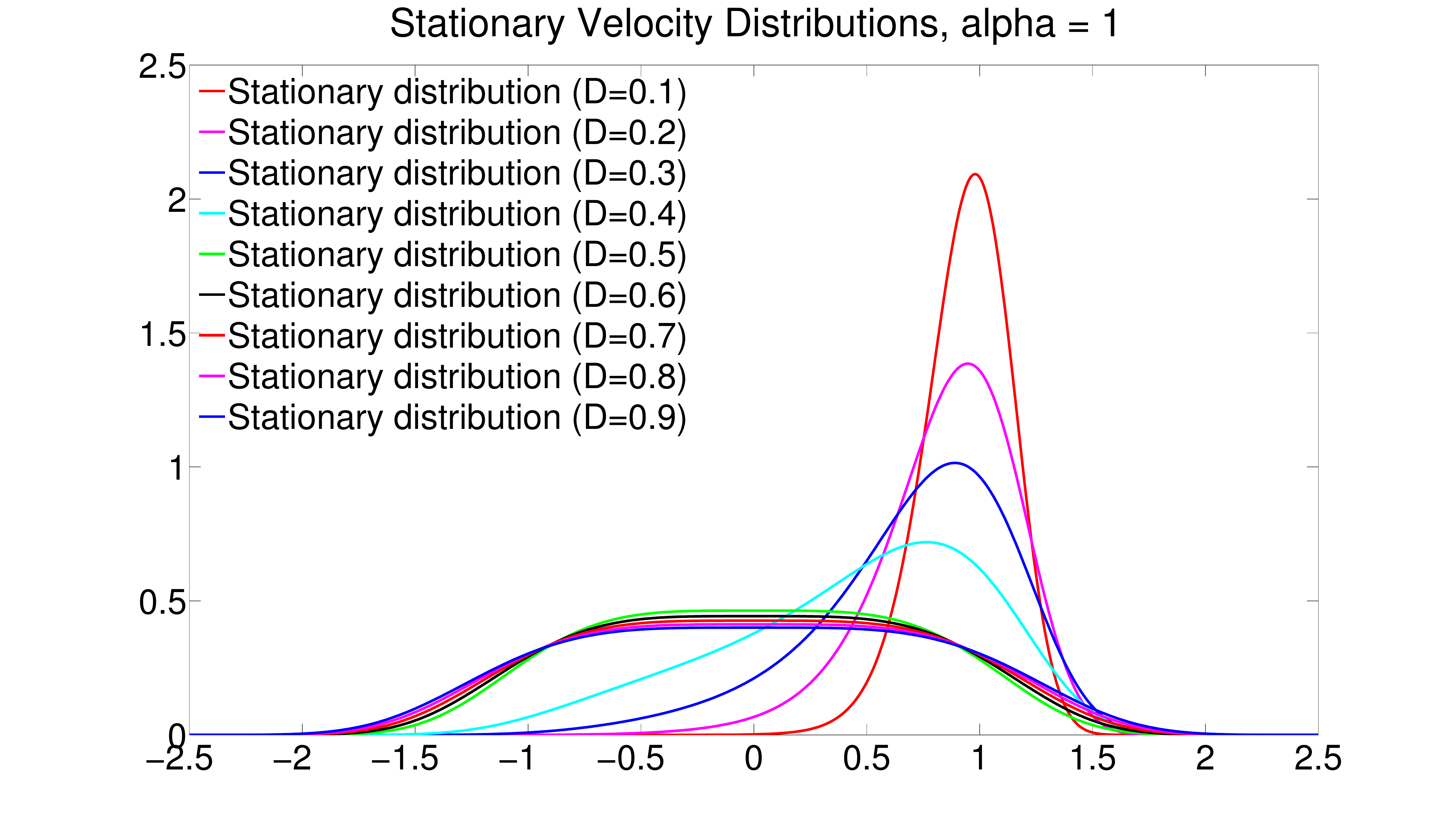}\\
 \includegraphics[width=8cm]{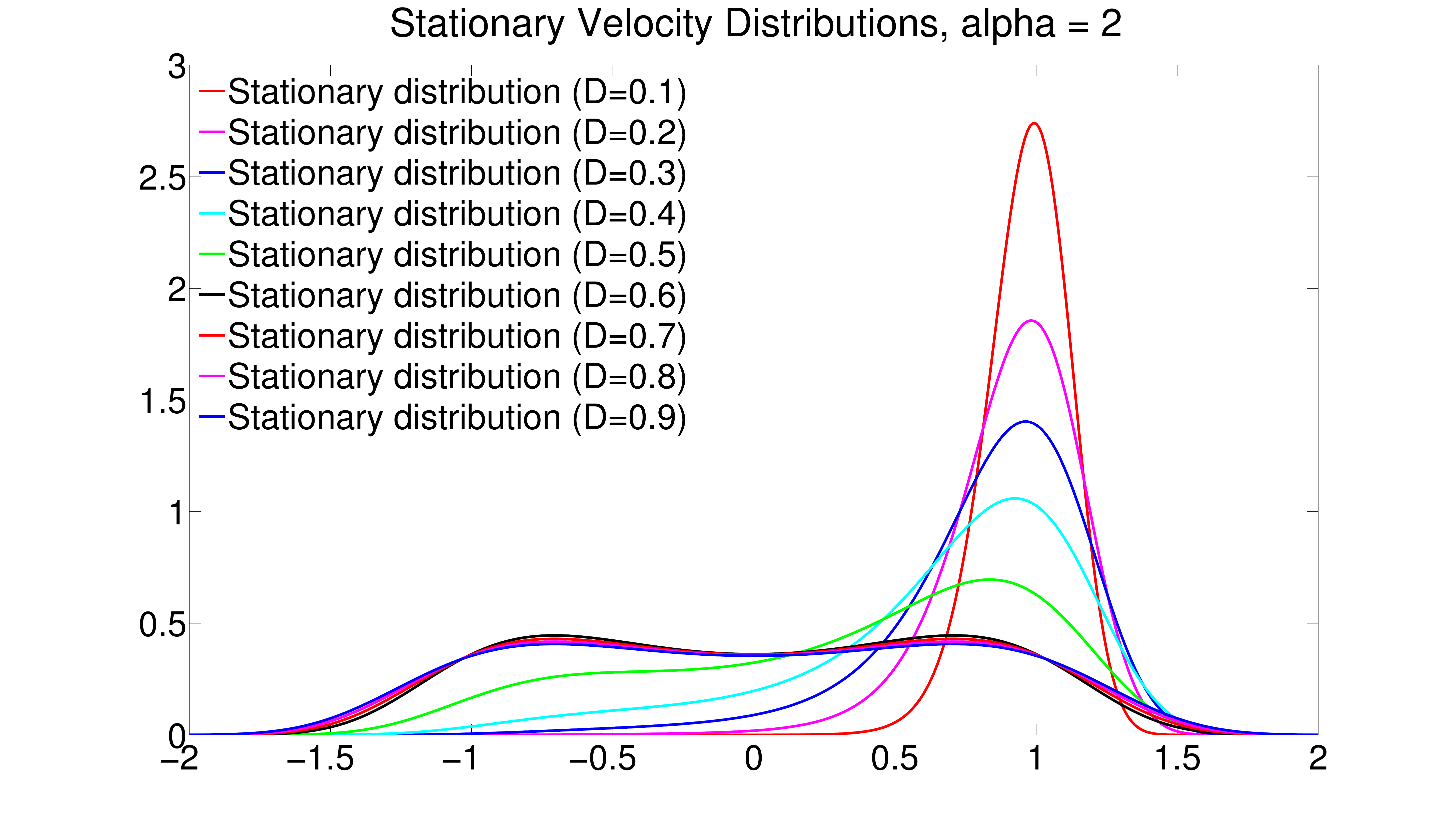}  
 \includegraphics[width=8cm]{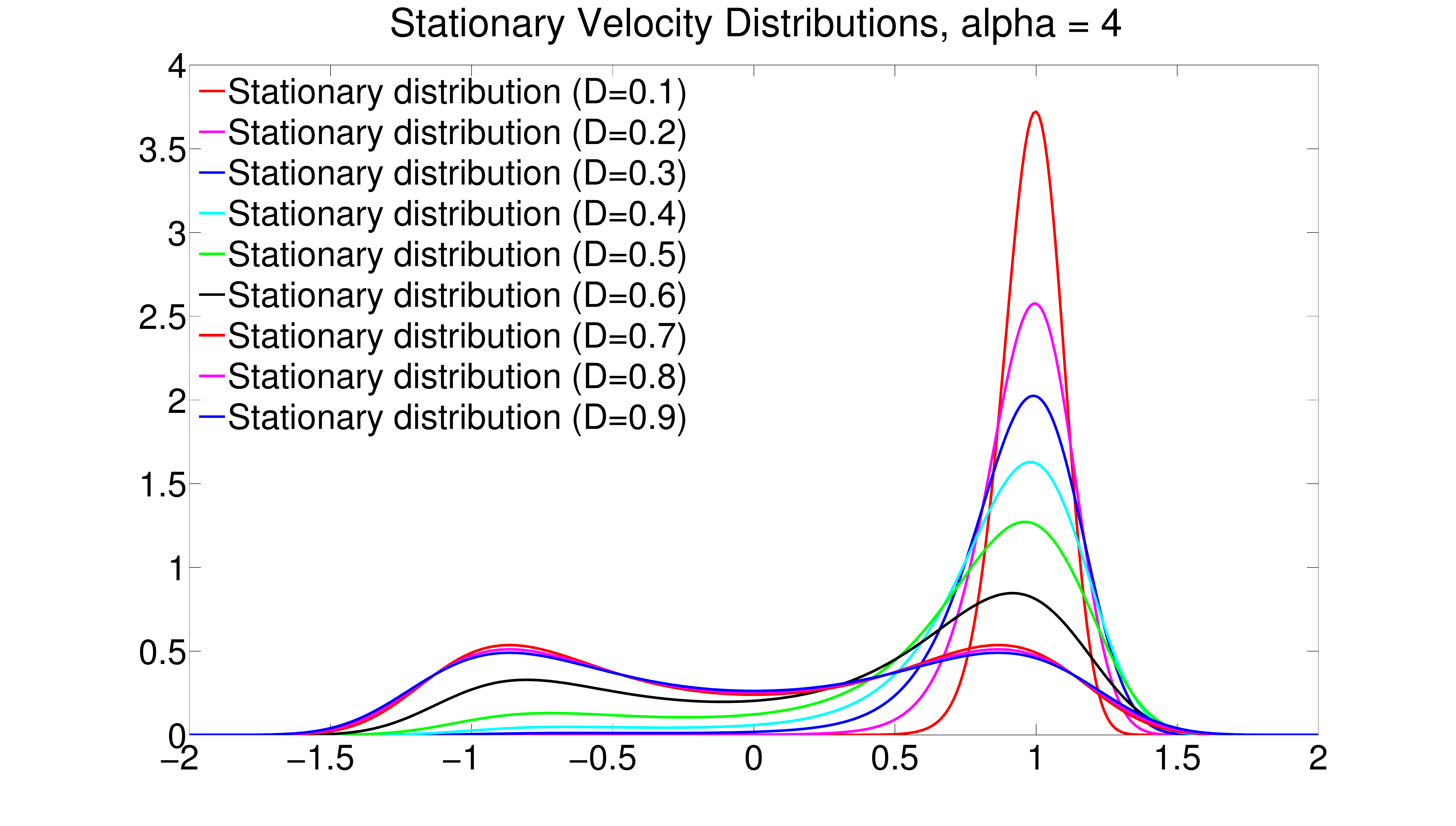}\\
 \includegraphics[width=8cm]{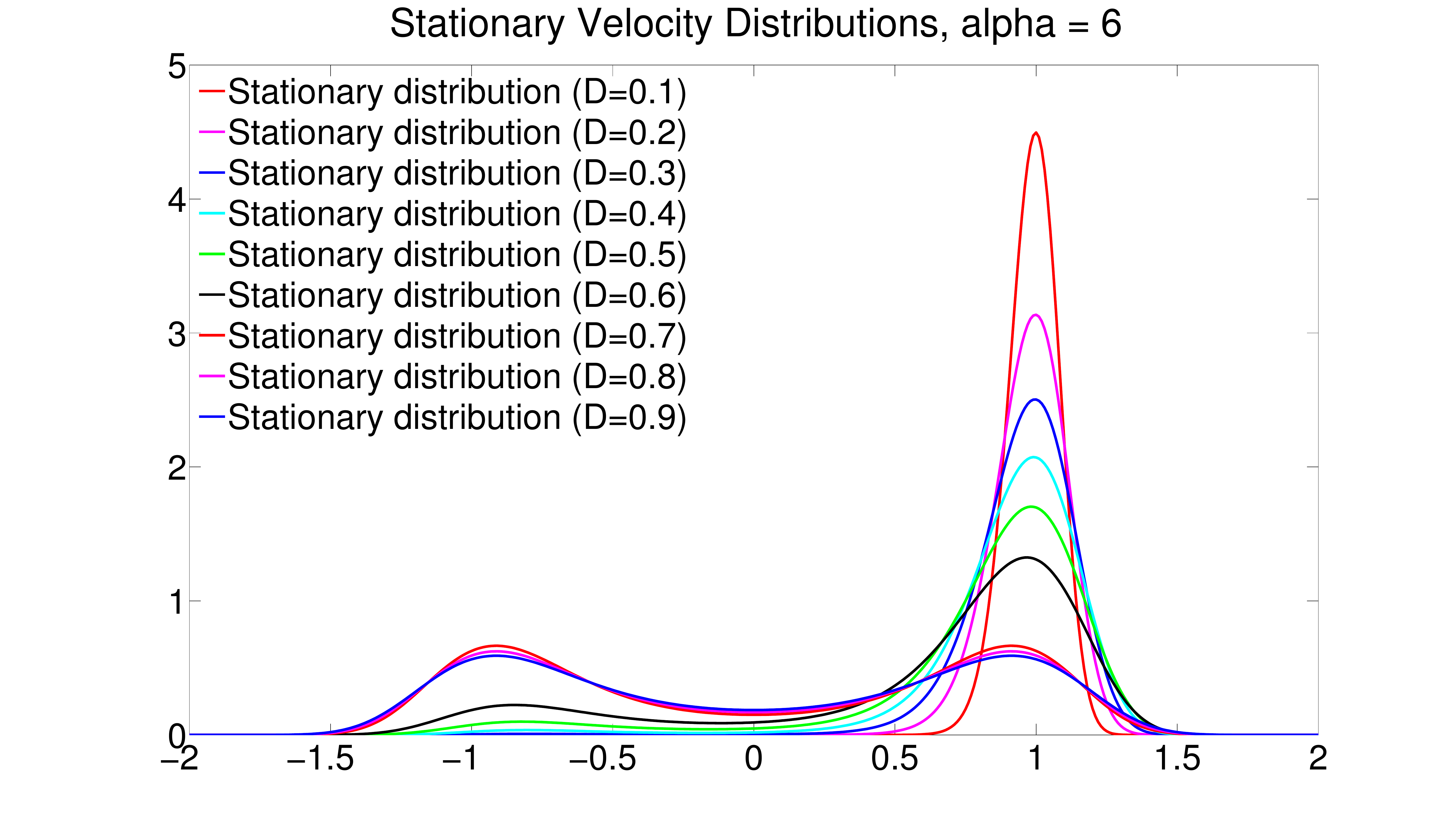}
 \end{center}
 \caption{Stationary distributions for various $\alpha$ and $D$ in the
   one-dimensional case.  In all of the figures, we plot the
   stationary distributions for several $D$ values; from left to
   right, we fix $\alpha$ to be $0.001$, $1$, $2$, $4$, and $6$,
   respectively.  The stationary distributions are computed by
   substituting the average velocity which we found numerically into
   the formula for the stationary distribution \eqref{stateq}.}
 \label{F:stationaryDists_alpha}
 \end{figure}

 With the formula for the steady-state and the stationary average
 velocities that we found numerically, we are able to consider also
 the shape of the stationary distributions in velocity space as
 $\alpha$ changes.  The results shown in Figure
 \ref{F:stationaryDists_alpha} are for the one-dimensional case,
 though this could easily be done for the two-dimensional case, as
 well.  We observe that the stationary distribution for $D$ beyond the
 critical threshold $D_c$ can take forms which are not Gaussian around
 $0$.  In fact, they can be double-peaked, with the peaks at $v=-1$
 and $v=1$.  We can see from Figure \ref{F:stationaryDists_alpha} that
 the preference for velocities of norm $1$ is not strong enough to
 create a double-peaked distribution for the case of small $\alpha$,
 but as $\alpha$ grows, there is a definitive preference for
 velocities with norm 1. Beyond the critical noise threshold, double
 peaks at velocities with norm one are clearly apparent even in cases
 in which $u(D)= 0$.

\subsection{Free energy and average velocity in 1D}

The particle simulations give us insight beyond indicating the
stability of the stationary states.  Using an ensemble of $100$ runs,
each with $10000$ particles and a time step of $0.001$, we are able to
construct a histogram approximating the velocity profile $f$ in one
dimension.  We used this to calculate the average velocity over time
and to calculate the free energy given in equation \eqref{freeen} and
plot it over the course of the simulation.  In Figure \ref{F:aveVel},
we show the average velocity over time of the particles in the
ensemble of runs. We can see that the average velocity initially dips
for all of the values of $D$, presumably as the particles align, and
then more slowly approaches the velocity that we know from figure
\ref{F:stability} to be the stationary value of the average velocity.

In Figure \ref{F:entropy}, we plot the evolution of the free energy
from the histograms of the particles.  We again see an initial swift
decline in the free energy, followed by a gradual decay of the free
energy for small values of $D$, as is expected. For larger values of
$D$, after this initial steep decay of the free energy we notice a
gradual \emph{increase}. Note that this particular method is not
designed to preserve the decay of the free energy. In fact, we observe
that our method initially undershoots the asymptotic value of the free
energy for large noise since the limiting free energy value of the
stationary state is accurately computed in view of the results in
Figure \ref{F:histAgainstTruth}. Free-energy decreasing deterministic
methods \cite{CCH} could be used to investigate this issue further.

\begin{figure}
\begin{center}
\includegraphics[width=8cm]{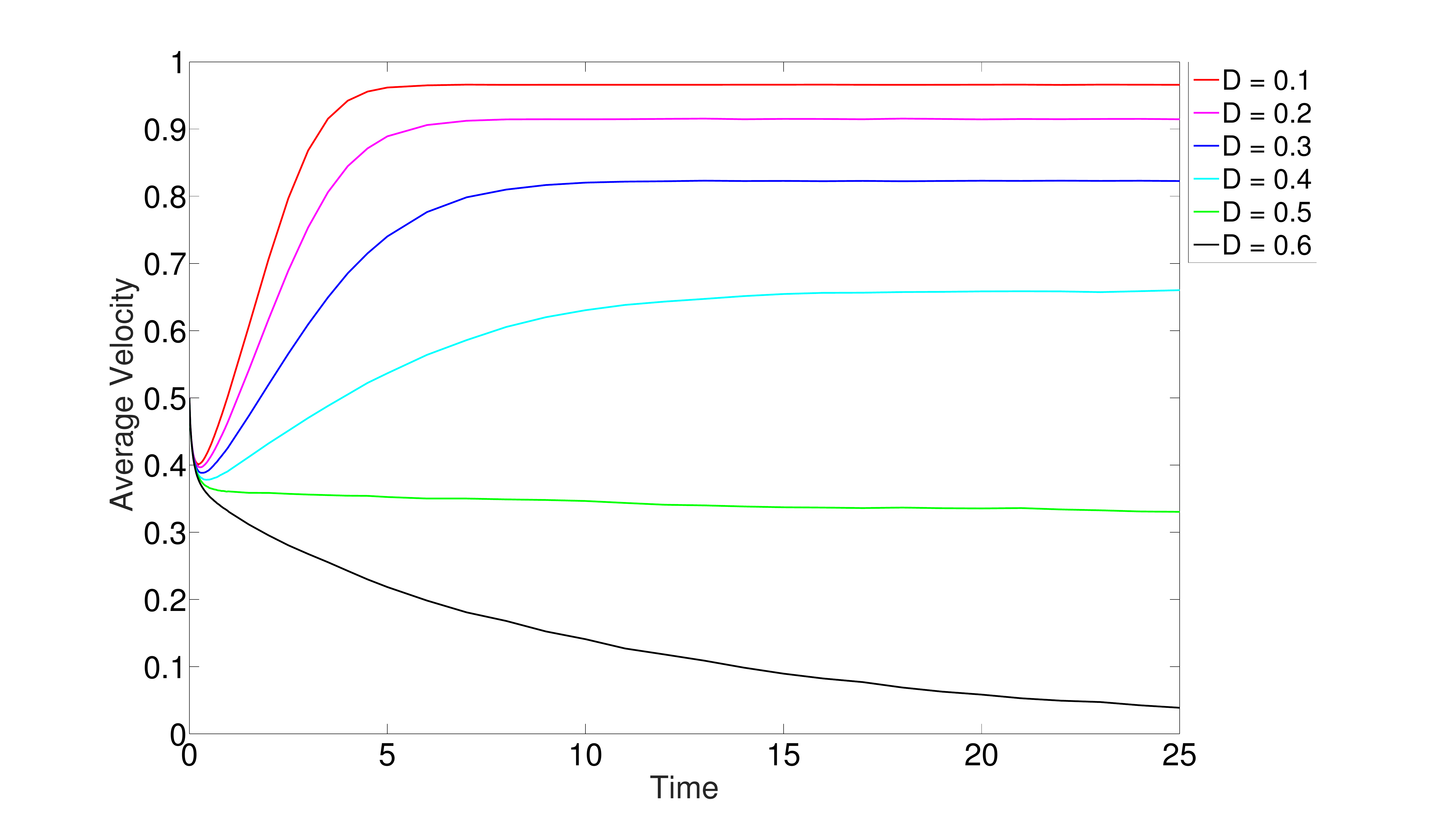}
\includegraphics[width=8cm]{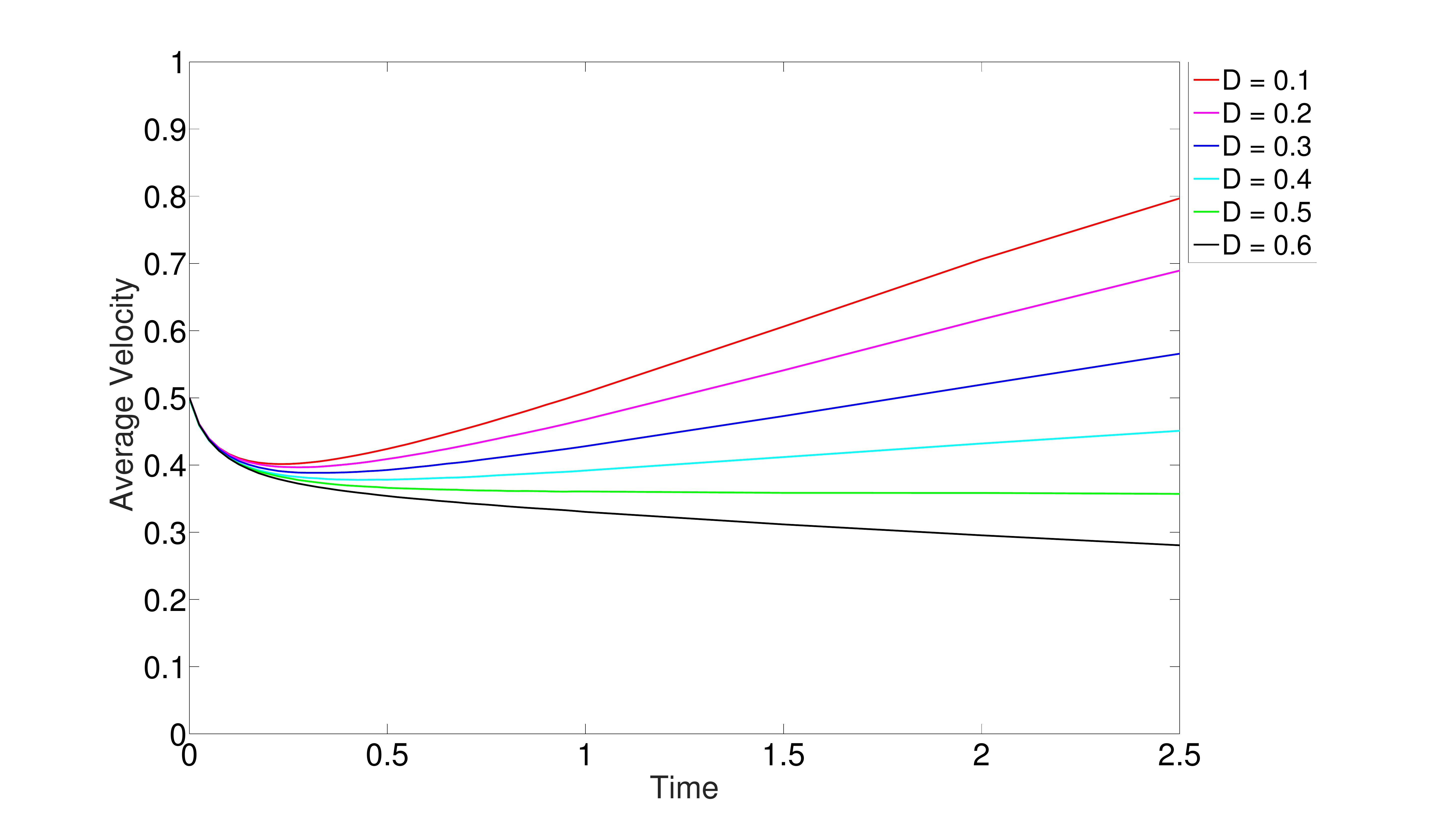}
\end{center}
\caption{One Dimension: The average velocity calculated for several values of $D$ with $\alpha = 2$. The left figure shows the evolution of the average velocity to time $25$, while the figure on the right focuses on the initial period from $t=0$ to $t=2.5$. Note the initial dip in average velocity for all values of $D$.}
\label{F:aveVel}
\end{figure}

\begin{figure}
\begin{center}
\includegraphics[width=8cm]{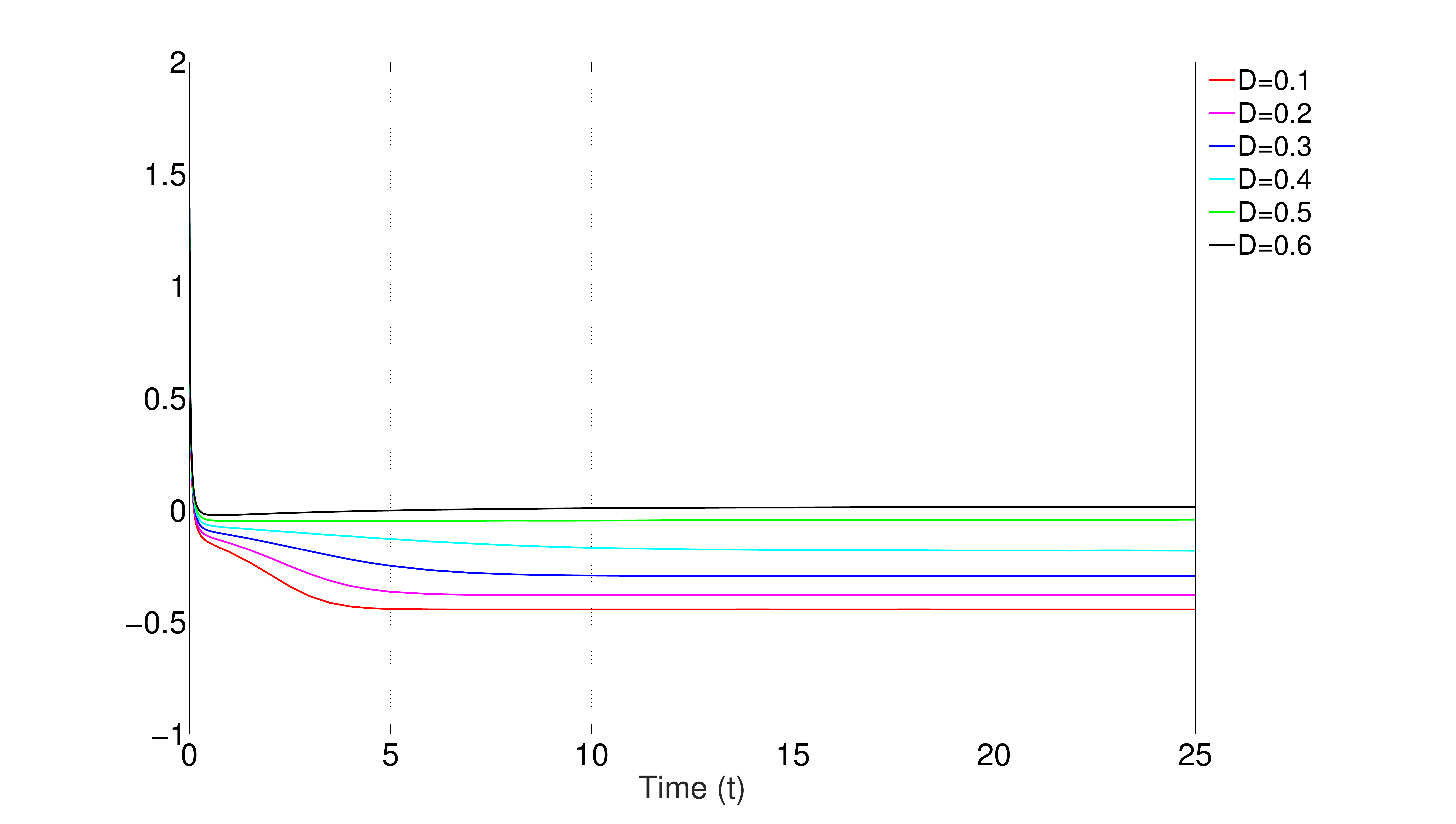}
\includegraphics[width=8cm]{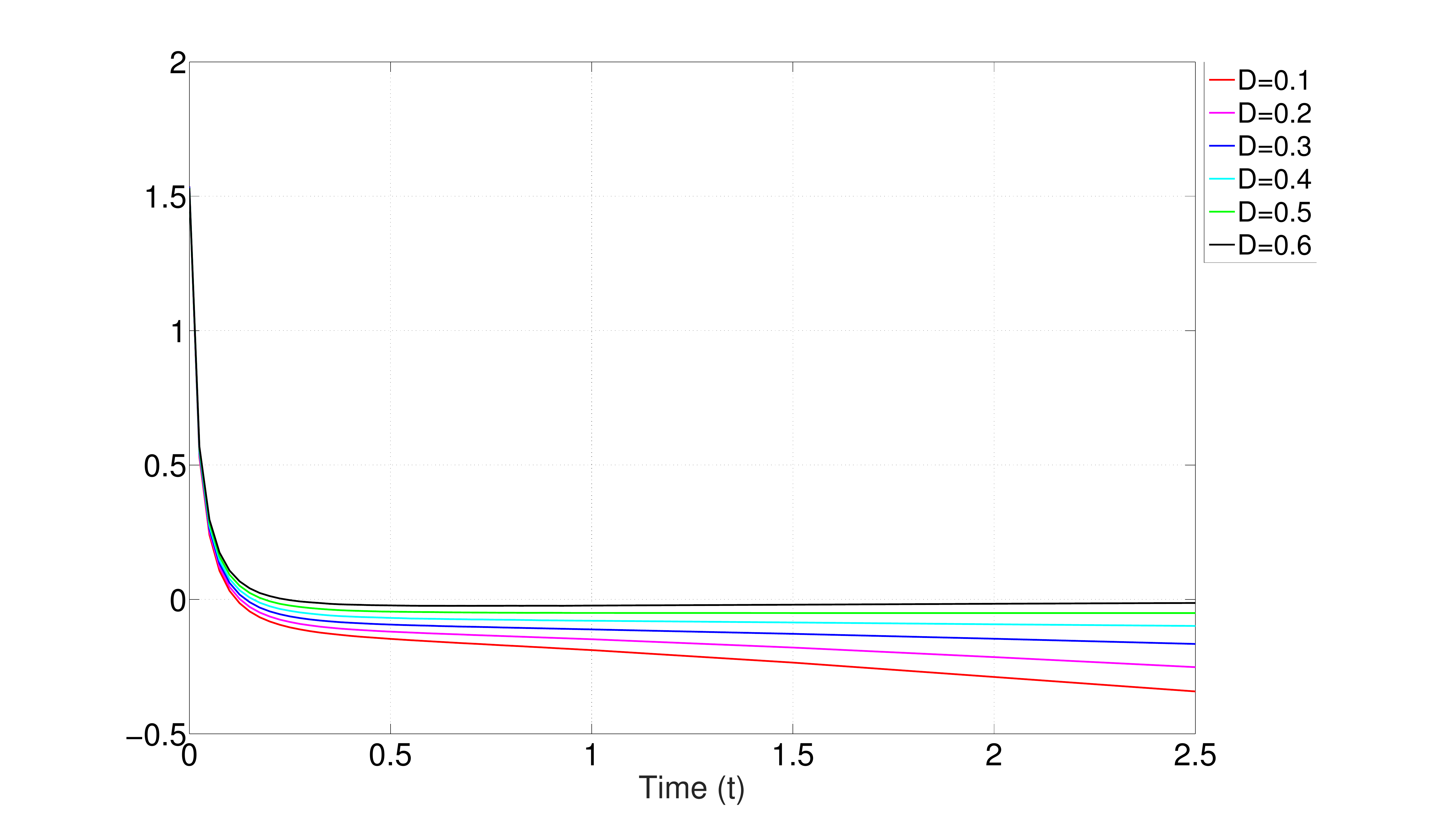}
\end{center}
\caption{Free Energy in One Dimension: free energy over time for
  varying values of $D$. The first plot shows the free energy up to
  time $25$, while the second shows the same plot but zoomed to focus
  on the initial period, where the free energy decreases rapidly for
  all values of $D$.}
\label{F:entropy}
\end{figure}


\section{Conclusion}\label{S:conclusion}

In this work we have studied the stationary solutions of a noisy
self-propelled Cucker-Smale model and proven than this model behaves
similarly as the Vicsek model, i.e. it exhibits a phase transition
when the noise intensity increases. For large noise intensity, we show
the existence of a single equilibrium with zero mean-velocity, showing
that no self-organized motion is possible. For small noise intensity,
we show the existence of a family of polarized equilibria parametrized
by a unit vector on the sphere. There are still a number of issues
left. For instance, so far, we cannot rule out the existence of other
families of equilibria for small noise, as our approach relies on a
local approach using the implicit function theorem. Also, several
phase transition points could exist with the emergence of other
branches of equilibria. Note that these circumstances are not
occurring in one dimension thanks to the results of Tugaut
\cite{Tugaut2013Convergence, Tugaut2013Phase,
  Tugaut2013Selfstabilizing,Tugaut2014Selfstabilizing}. Another
direction is to take advantage of the polarized equilibria to develop
a hydrodynamic model in a similar spirit as \cite{dm2008} for the
Vicsek model. Here again, we expect that the absence of analytical
formulas for the mean velocity of the equilibria will generate
additional difficulties. Numerically, it could be interesting, but
challenging to compare the dynamics resulting from the Vicsek model
and the self-propelled Cucker-Smale model with high precision. The
assumption of constant speed in the Vicsek model has often been
disputed and this comparison would help determine what consequences
this assumption has on the results. A similar comparison could be made
at the level of the hydrodynamic models and would be equally useful.


\appendix

\section{Calculations on the asymptotic behavior of $H$ as $D \to 0$}
\label{sec:appendix}

We gather here some computations needed in order to calculate the
value of several limits of $H$ as $D \to 0$ used in Section
\ref{S:anyDimension}, such as
$\lim_{D \to 0} \frac{\partial H}{\partial D}(u,D)$ for $u=1$. We
follow the notation of Section \ref{S:anyDimension}.

The value of the following Gaussian integrals is used in the
calculations below:
\begin{equation*}
  \int_{\R} e^{-r^2} r^{2n} \d r = \sqrt{\pi} \frac{(2n-1)!!}{2^{n}}
  \qquad (n \in \NN),
\end{equation*}
from which we readily see that
\begin{gather*}
  \ird z_1^2 e^{-|z|^2} \d z = \frac12 \pi^{N/2},
  \qquad
  \ird z_1^2 z_2^2 e^{-|z|^2} \d z = \frac14 \pi^{N/2},
  \\
  \ird z_1^4 z_2^2 e^{-|z|^2} \d z = \frac38 \pi^{N/2},
  \quad
  \ird z_1^1 z_2^2 z_3^2 e^{-|z|^2} \d z = \frac18 \pi^{N/2},
  \\
  \ird z_1^4 e^{-|z|^2} \d z = \frac34 \pi^{N/2},
  \qquad
  \ird z_1^6 e^{-|z|^2} \d z = \frac{15}{8} \pi^{N/2}.
\end{gather*}
Using Lemma \ref{lem:Laplace-basic} we then obtain
\begin{subequations}
  \label{eq:mi}
  \begin{gather}
    m_2 := \ird z_1^2 e^{-\overline{Q}_1(z)} \,dz = (2 \pi)^{N/2}
    (1+2\alpha)^{-3/2},
    \\
    m_4 := \ird z_1^4 e^{-\overline{Q}_1(z)} \,dz
    = 3 (2 \pi)^{N/2} (1+2\alpha)^{-5/2},
    \\
    m_{22} := \ird z_1^2 z_2^2 e^{-\overline{Q}_1(z)} \,dz =
    (2 \pi)^{N/2} (1+2\alpha)^{-3/2},
    \\
    m_{222} := \ird z_1^2 z_2^2 z_3^2 e^{-\overline{Q}_1(z)} \,dz =
    (2 \pi)^{N/2} (1+2\alpha)^{-3/2},
    \\
    m_{24} := \ird z_1^2 z_2^4 e^{-\overline{Q}_1(z)} \,dz
    = 3 (2 \pi)^{N/2} (1+2\alpha)^{-3/2},
    \\
    m_{42} := \ird z_1^4 z_2^2 e^{-\overline{Q}_1(z)} \,dz
    =  3 (2 \pi)^{N/2} (1+2\alpha)^{-5/2},
    \\
    m_6 := \ird z_1^6 e^{-\overline{Q}_1(z)} \,dz = 15 (2 \pi)^{N/2}
    (1+2\alpha)^{-7/2},
  \end{gather}
\end{subequations}
where we recall that $\overline{Q}_1(z) = \frac12 |z|^2+\alpha z_1^2$
was defined in \eqref{eq:a0Q1}.

\paragraph{Value of $c_1(1)$.}

\begin{align*}
  c_1(1) &= -\alpha \ird z_1^2 |z|^2 e^{-\overline{Q}_1(z)} \,dz
  = -\alpha (m_4 + (N-1)m_{22})
  \\
& = -\alpha (2\pi)^{N/2} \left(
    3 (1+2\alpha)^{-5/2} + (N-1) (1+2\alpha)^{-3/2}
  \right)
  \\
  &=
  -\alpha (2\pi)^{N/2} (1+2\alpha)^{-5/2}
  \left(N + 2 + 2 (N-1) \alpha
  \right).
\end{align*}

\paragraph{Value of $k_1(1)$.}

We have
\begin{equation*}
  k_1(1) = \alpha \left(
    \ird z_1^2 |z|^2 e^{-\overline{Q}_1(z)} \d z
    - \ird z_1^2 |z|^2 \overline{Q}_1(z) e^{-\overline{Q}_1(z)} \d z
  \right)
  =: \alpha (I_1 - I_2).
\end{equation*}
We calculate these integrals separately using the values in \eqref{eq:mi}:
\begin{equation*}
  I_1 =
  \ird z_1^2 |z|^2 e^{-\overline{Q}_1(z)} \,dz
  = m_4 + (N-1) m_{22}
  = (2\pi)^{N/2} \left(
    3(1+2\alpha)^{-5/2} + (N-1)(1+2\alpha)^{-3/2}
  \right).
\end{equation*}
\begin{align*}
  I_2 &=
  \ird z_1^2 |z|^2 \overline{Q}_1(z) e^{-\overline{Q}_1(z)} \,dz
  \\
&  =
  \left(\frac12 + \alpha\right) m_6
  +(N-1)(1+\alpha) m_{42}
  + \frac{N-1}{2} m_{24}
  + \frac{(N-1)(N-2)}{2} m_{222}
  \\
  &=
  (2\pi)^{N/2}
  \left(
    \left( \frac{15}{2} + 3(N-1)(1+\alpha)
    \right)
    (1 + 2 \alpha)^{-5/2}
    + \frac{N^2-1}{2} (1 + 2 \alpha)^{-3/2}
  \right).
\end{align*}
Finally, we get
\begin{align*}
  \frac{1}{(2\pi)^{N/2} \alpha} k_1(1)
  &=
  \left( - \frac{9}{2} - 3(N-1)(1+\alpha)
  \right)
  (1 + 2 \alpha)^{-5/2}
  - \frac{(N-1)^2}{2}
  (1 + 2 \alpha)^{-3/2}
  \\
  &=
  (1 + 2 \alpha)^{-5/2}
  \left(
    - \frac{9}{2} - 3(N-1)(1+\alpha)
    - \frac{(N-1)^2}{2}
    (1 + 2 \alpha)
  \right).
\end{align*}

\paragraph{Value of $\partial H / \partial D$ as $D \to 0$.}

Using this with \eqref{eq:c01}, \eqref{eq:c11} and \eqref{eq:k11} we have
$$
  c_0(1) k_1(1) - c_1(1)k_2(1)
  =
  -\frac{(2\pi)^{N}\alpha}{(1+2\alpha)^3}
    \left(
           3+(N-1) (1 + 2 \alpha)
  \right).
$$


\section*{Acknowledgments}

A. Barbaro was supported by the NSF through grant
No. DMS-1319462. J.~A.~Ca\~nizo was supported by the Spanish project
MTM2014-52056-P and the Marie-Curie CIG grant KineticCF.
J.~A. Carrillo was partially supported by the project
MTM2011-27739-C04-02 DGI (Spain), from the Royal Society by a Wolfson
Research Merit Award and by the EPSRC grant EP/K008404/1. PD acknowledges
support from EPSRC under grant ref: EP/M006883/1, from the Royal Society
and the Wolfson Foundation through a Royal Society Wolfson Research
Merit Award. PD is on leave from CNRS, Institut de Math\'ematiques de
Toulouse, France.

\bibliographystyle{plain}
\bibliography{bibliography}

\end{document}